\documentclass[bj]{imsart}

\RequirePackage{amsthm,amsmath,amsfonts,amssymb}
\RequirePackage[numbers]{natbib}

\RequirePackage{color}
\RequirePackage{graphicx}
\RequirePackage{microtype}

\startlocaldefs

\numberwithin{equation}{section}

\theoremstyle{plain}
\newtheorem{corollary}{Corollary}
\newtheorem{proposition}{Proposition}
\newtheorem{theorem}{Theorem}
\theoremstyle{remark}
\newtheorem{definition}{Definition}
\newtheorem{remark}{Remark}
\newtheorem*{notation}{Notation}

\def\N{\mathbb{N}}
\def\R{\mathbb{R}}
\def\Z{\mathbb{Z}}
\def\i{\mathrm i}
\def\d{\mathrm d}
\def\e{\mathrm e}
\def\E{\mathrm E}

\newcommand{\nn}{\nonumber}
\newcommand{\noi}{\noindent}

\def\1{{\bf 1}}
\def\0{{\bf 0}}

\newcommand{\mbe}{\boldsymbol{e}}
\newcommand{\mbH}{\boldsymbol{H}}

\newcommand{\mbs}{\boldsymbol{s}}
\newcommand{\mbt}{\boldsymbol{t}}
\newcommand{\mbu}{\boldsymbol{u}}

\newcommand{\mbx}{\boldsymbol{x}}
\newcommand{\mby}{\boldsymbol{y}}

\def\norm{|\!|}

\def\limd{
	\begin{array}[t]{c}
		\stackrel{{\rm d}}{\to}
	\end{array}}

\def\eqd{
	\begin{array}[t]{c}
		\stackrel{{\rm d}}{=}
\end{array}}

\def\neqd{
	\begin{array}[t]{c}
		\stackrel{{\rm d}}{\neq}
\end{array}}

 \def\limfdd{
 	\begin{array}[t]{c}
 		\stackrel{{\rm fdd}}{\to}
 \end{array}}

\def\eqfdd{
	\begin{array}[t]{c}
		\stackrel{{\rm fdd}}{=}
	\end{array}}

\def\neqfdd{
	\begin{array}[t]{c}
		\stackrel{{\rm fdd}}{\neq}
	\end{array}}

\endlocaldefs

\begin{document}

\begin{frontmatter}
\title{Local scaling limits of L\'evy driven fractional random fields}
\runtitle{Local scaling limits of L\'evy driven fractional RFs}

\begin{aug}
\author[A]{\fnms{Vytaut\.e} \snm{Pilipauskait\.e}\ead[label=e1,mark]{vytaute.pilipauskaite@gmail.com}}
\and
\author[B]{\fnms{Donatas} \snm{Surgailis}\ead[label=e2]{donatas.surgailis@mif.vu.lt}}

\address[A]{University of Luxembourg, Department of Mathematics,
	6 Avenue de la Fonte, 4364 Esch-sur-Alzette, Luxembourg, \printead{e1}}

\address[B]{Vilnius University, Faculty of Mathematics and Informatics, Naugarduko 24, 03225 Vilnius, Lithuania, \printead{e2}}
\end{aug}

\begin{abstract}
We obtain a complete description of local anisotropic scaling limits for a class of fractional  random fields $X$ on ${\mathbb{R}}^2$ written as stochastic integral with respect to infinitely divisible random measure.
The scaling procedure involves increments of $X$ over points the distance between which in the horizontal and vertical  directions shrinks as $O(\lambda) $ and $O(\lambda^\gamma)$ respectively as $\lambda \downarrow 0$, for some $\gamma>0$.
We consider two types of increments  of $X$: usual increment and rectangular increment, leading to the respective concepts of $\gamma$-tangent and $\gamma$-rectangent random fields.
We prove that for above $X$ both  types of local scaling  limits exist for any $\gamma>0$ and undergo a transition, being independent of $\gamma>\gamma_0$ and $\gamma<\gamma_0$, for some $\gamma_0>0$; moreover,  the `unbalanced' scaling limits ($\gamma\ne\gamma_0$) are  $(H_1,H_2)$-multi self-similar with one of $H_i$, $i=1,2$, equal to $0$ or $1$.
The paper extends Pilipauskait\.e  and Surgailis (2017) and Surgailis (2020) on large-scale anisotropic scaling of random fields on ${\mathbb{Z}}^2$ and Benassi et al.\ (2004) on $1$-tangent limits of isotropic fractional L\'evy random fields.
\end{abstract}


\end{frontmatter}


\section{Introduction}

The present paper discusses local anisotropic scaling behavior of a class of fractional type infinitely divisible random fields (RFs) on $\R^2 $.
This behavior is characterized by limits of shrinking increments of RF under anisotropic scaling, where anisotropy is due to the fact that the `horizontal' and `vertical dimensions' of the increment tend to 0 with $\lambda \downarrow 0$ at different rates $\lambda $ and $\lambda^\gamma$ for any given $\gamma >0$.

Given a RF $X = \{X(\mbt), \, \mbt \in \R^2 \}$, the two basic types of its increment at a point $\mbt_0 = (t_{01}, t_{02}) \in \R^2 $ are
\emph{(ordinary) increment}  $X(\mbt_0 + \mbt) - X(\mbt_0)$ and \emph{rectangular increment}
\begin{equation}\label{Xinc}
	X((\boldsymbol{t}_0, \mbt_0 + \boldsymbol{t}]) := X(t_{01}+t_1, t_{02} + t_2)-X(t_{01},t_{02}+ t_2)- X(t_{01} + t_1, t_{02}) + X(t_{01},t_{02})
\end{equation}
for $\mbt=(t_1,t_2) \in \R^2_+$.
These two notions of increment give rise to different notions of RF with
\emph{stationary increments} and \emph{stationary rectangular increments} \cite{bass2012}.
For $\gamma >0$, $\lambda >0$, let $\Gamma = \operatorname{diag}(1, \gamma)$, $\lambda^\Gamma = \operatorname{diag}(1, \lambda^\gamma)$.
With $\lambda^\Gamma \mbt = (\lambda t_1, \lambda^\gamma t_2)$, (anisotropic) local scaling limits of RF $X$ at $\mbt_0$ can be defined as the limits (in the sense of weak convergence of finite-dimensional distributions) as $\lambda \downarrow 0$:
\begin{align}
	d_{\lambda,\gamma}^{-1} (X(\mbt_0 + \lambda^\Gamma \mbt) - X(\mbt_0)) &\limfdd T_\gamma(\mbt), \label{Tlim}\\
	d_{\lambda,\gamma}^{-1} X((\mbt_0, \mbt_0 + \lambda^{\Gamma}\mbt])&\limfdd V_\gamma(\mbt), \label{Vlim}
\end{align}
where $d_{\lambda,\gamma} \downarrow 0$ is a normalization, the latter being generally different for
\eqref{Tlim} and \eqref{Vlim}.
Here and below, we suppress the dependence of these scaling limits on $\mbt_0$ also because in our theorems they actually do not depend on it.
In the case of isotropic scaling $\gamma =1 $ the limit $T_1$ in \eqref{Tlim} (called the \emph{tangent RF}) was introduced in \cite{falc2002} with $\limfdd$ replaced by a stronger functional convergence.
The existence of (non-trivial) tangent RF is also termed \emph{local asymptotic self-similarity} \cite{benn2004, cohen2012, cohen2013} since the tangent RF is self-similar  \cite{falc2002}.
In fact, following \cite{falc2002}, \cite[Prop.~2.1]{ps2016}, one can prove that under mild additional conditions all scaling limits in \eqref{Tlim}--\eqref{Vlim} satisfy the \emph{$(H,\gamma)$-SS} (self-similarity) property:
\begin{eqnarray}\label{gammass}
	U(\lambda^\Gamma \mbt) \eqfdd \lambda^{H} U(\mbt), \qquad \forall \lambda >0,
\end{eqnarray}
with some $H = H(\gamma)>0$; moreover, the normalization  $d_{\lambda,\gamma}$ is regularly varying with exponent $H $ as $\lambda \downarrow 0$.
Superseding  this terminology we call $T_\gamma $ in \eqref{Tlim} the \emph{$\gamma$-tangent RF} and $V_\gamma $ in \eqref{Vlim} the \emph{$\gamma$-rectangent RF} (`rectangent' as the abridge for `rectangular tangent').
Note that for $\gamma =1 $ \eqref{gammass} yields the (usual) SS property for RF indexed by $\R^2 $ or $\R^2_+$, see \cite{lamp1962, falc2002, samo1994}, for general $\gamma >0$  \eqref{gammass} is a particular case of operator scaling property discussed in \cite{bier2007}.

The RFs $X$ for which $\gamma$-tangent and  $\gamma$-rectangent RF ($\gamma>0 $ arbitrary) are identified in this paper are written as stochastic integrals
\begin{equation}\label{def:X}
	X(\boldsymbol{t})  :=  \int_{\R^2} \big\{g (\boldsymbol{t}- \boldsymbol{u}) - g^0_1 ((t_1,0)- \boldsymbol{u})  -  g^0_2 ((0,t_2)- \boldsymbol{u}) +
	g^0_{12}(- \boldsymbol{u})\big\}
	M(\d \boldsymbol{u}),  \quad \boldsymbol{t} = (t_1,t_2) \in \R^2_+,
\end{equation}
where $M$ is an infinitely divisible random measure   (also called a L\'evy basis) on $\R^2$ and $g$, $g^0_1$, $g^0_2$, $g^0_{12}$ are deterministic functions satisfying some conditions guaranteeing the existence of \eqref{def:X}.
For reasons explained below, we call $X$ in \eqref{def:X} a \emph{L\'evy driven fractional RF}.
It follows from \eqref{def:X} that RF $X$ has stationary rectangular increments which do not depend on the `initial' functions  $g^0_1$, $g^0_2$, $g^0_{12}$, viz.,
\begin{equation}\label{Xinc1}
	X((\boldsymbol{0}, \boldsymbol{t}])
	=  \int_{\R^2} g ((-\mbu, \boldsymbol{t}- \boldsymbol{u}])
	M(\d \boldsymbol{u}), \qquad \mbt \in \R^2_+,
\end{equation}
where $g ((-\mbu, \boldsymbol{t}- \boldsymbol{u}]) := g(t_1-u_1,t_2-u_2)-g(-u_1,t_2-u_2)- g(t_1-u_1,-u_2) + g(-u_1,-u_2) $ in accordance with the
notation in \eqref{Xinc}.
Similarly, if $g^0_1=g^0_2 = 0$ then $X$ in \eqref{def:X} has stationary (ordinary) increments.
In both cases, the scaling limits in \eqref{Tlim}--\eqref{Vlim} (provided they exist) do not depend on $\mbt_0$ and depend on scaling properties of kernel $g$ and the L\'evy basis $M$  specified in Assumptions (G)$_\alpha$ and (M)$_\alpha$ below  (roughly, the last assumption means that the small-scale behavior of $M$  is $\alpha$-stable with $0< \alpha \le 2$).
The class of RFs in  \eqref{def:X} is quite large and contains many fractional RFs studied in \cite{benn2004, cohen2012, take1991, jons2013, hansen2013} and elsewhere.
\eqref{def:X} also constitute a natural spatial  generalization of L\'evy driven moving average processes with one-dimensional time studied in \cite{bass2017}.

The main results of this paper can be summarised as follows.
We prove that for a class of fractional L\'evy driven RFs in \eqref{def:X} the $\gamma$-rectangent limits $V_\gamma$ exist for any $\gamma>0$ and are $\alpha$-stable RFs; moreover, the limit family $\{V_\gamma, \, \gamma>0\}$ exhibits a `scaling transition'  in the sense that there exists $\gamma_0 >0$ such that $V_\gamma = V_+ $ (respectively, $V_\gamma = V_- $) do not depend on $\gamma > \gamma_0$ (respectively, on $\gamma < \gamma_0$) and  $V_+ \neqfdd cV_- $  for any $c>0$;
moreover, the `unbalanced'  $\gamma$-rectangent limits $V_\pm$ are \emph{$(H_1,H_2)$-multi self-similar} (MSS) RFs (see \eqref{mss}) with \emph{one of the self-similarity parameters $H_i$, $i=1,2$, equal $1$ or $0$}.
We also prove somewhat similar although more straightforward results, including a `scaling transition', about $\gamma$-tangent limits $T_\gamma$ for a related class of  fractional L\'evy driven RFs in \eqref{def:X}.

Related trichotomy of the scaling behavior was reported in large-scale anisotropic scaling for several classes of long-range dependent (LRD) planar RF models, with rectangular increment replaced by a sum or integral of the values on large rectangle with sides increasing at different rates $\lambda$ and $\lambda^\gamma$ as $\lambda \to \infty$ for any given $\gamma>0$; see  \cite{ps2015, ps2016, pils2016, pils2017, pils2020, sur2020}. See also \cite{bier2017}.  
In the above works, this trichotomy was  termed the \emph{scaling transition}, with $V_\pm $ the \emph{unbalanced} and $V_{\gamma_0}$ the \emph{well-balanced} scaling limits.
The present paper can be regarded as a continuation of the above research and we use the same terminology in the case of the `small-scale' limits in \eqref{Tlim}--\eqref{Vlim}.
As noted above the unbalanced limits in \eqref{Tlim}--\eqref{Vlim} have a very particular dependence structure being $(H_1,H_2)$-MSS RFs with one of $H_i$, $i=1,2$, equal $0$ or $1$.
Following  \cite{gent2007} we call a RF $V = \{V(\mbt), \, \mbt \in \R^2_+ \}$ $(H_1,H_2)$-MSS with parameters  $H_i \ge 0$, $i=1,2$, if
\begin{equation} \label{mss}
	V(\lambda_1 t_1, \lambda_2 t_2) \eqfdd  \lambda^{H_1}_1 \lambda_2^{H_2} V (\boldsymbol{t}), \qquad \forall \lambda_1 >0, \ \forall \lambda_2 >0.
\end{equation}
(We note that while \cite{gent2007} assume $H_i >0$, $i=1,2$, relation \eqref{mss} makes sense when $H_1 \wedge H_2 = 0$ as well.)
The `classical' case of $(H_1,H_2)$-MSS RF is \emph{Fractional Brownian Sheet} (FBS) $B_{H_1,H_2} $ defined as a Gaussian process on $\R^2_+$ with zero mean and the covariance function:
\begin{equation}\label{FBScov}
	\E B_{H_1,H_2}(\mbt)   B_{H_1,H_2} (\mbs) = {\textstyle\frac{1}{4}} \prod_{i=1}^2 (t_i^{2H_i} +  s_i^{2H_i} - |t_i-s_i|^{2H_i}),
\end{equation}
for $\mbt = (t_1,t_2)\in \R^2_+$, $\mbs = (s_1,s_2) \in \R^2_+$.
The parameters $H_i$, $i=1,2$, of FBS usually take values in the interval $(0,1]$ or even $(0,1)$, see \cite{aya2002}; the extension to $H_i \in [0,1]$ was defined in \cite{sur2020} from \eqref{FBScov} by continuity as $H_i \downarrow 0$, $i=1,2$, see Definition \ref{defFBS}.
FBS with $H_1 \wedge H_2 = 0$ are very unusual and extremely singular objects and their appearance in limit theorems is surprising \cite{sur2020}.
We note the FBS with $(H_1,H_2) = (0, \frac{1}{2}) $ or $(\frac{1}{2},0)$ as anisotropic partial sums  limits of linear RFs on $\Z^2 $ with negative dependence and edge effects were obtained in \cite{sur2020}; in the case of LRD  RFs on $\Z^2 $ unbalanced Gaussian limits were proved to be FBS  with at least one of  $H_i$, $i=1,2$, equal $1/2$ or $1$ \cite{pils2017, pils2020}.

Let us describe the results of the present paper informally in more detail. Rigorous formulations are given in Sec.~\ref{sec5} and \ref{sec6}.
We assume that $g$ in \eqref{def:X} and \eqref{Xinc1} features  a power-law behavior at the origin with possibly different exponents along the coordinate axes, more precisely, as $\mbt \to \0$,
\begin{equation} \label{glim}
	g(\mbt) \sim g_0(\mbt) :=\rho(\mbt)^\chi L(\mbt),
	\qquad \text{where } \rho(\mbt):=  |t_1|^{q_1} + |t_2|^{q_2},
\end{equation}
$q_1>0$, $q_2>0$, $Q:=\frac{1}{q_1}+\frac{1}{q_2}$, $\chi \ne 0$ are parameters, and $L(\mbt)$, $\mbt \in \R^2_0$, is a (generalized invariant) function satisfying some boundedness and regularity conditions.
For $L(\mbt) \equiv 1 $, the kernel $g_0(\mbt) = \rho(\mbt)^\chi  $ vanishes or explodes at the origin $\mbt = \boldsymbol{0}$ depending on the sign of the parameter $\chi $ suggesting a different scaling behavior in \eqref{Vlim} in the  cases $\chi >0$ and $\chi <0$.
Actually, the limit results essentially depend on the two parameters
\begin{equation}\label{pp}
p_i := q_i(Q-\chi) >0,  \qquad i=1,2,
\end{equation}
alone, making the parametrization $p_i$, $i=1,2$, in \eqref{pp} more convenient  than $\chi$, $q_i$, $i=1,2$, in \eqref{glim}.
The parameters  \eqref{pp} satisfy
\begin{eqnarray}\label{pin}
	&\frac{\alpha}{1+\alpha} < P < \alpha, \qquad \text{where } P := \frac{1}{p_1} + \frac{1}{p_2}.
\end{eqnarray}
Note for $p_i$, $i=1,2$, in \eqref{pp} satisfying \eqref{pin} $\chi >0 $ is equivalent to $P > 1 $.
For $c_i >0$, $i=1,2$, denote
\begin{eqnarray}\label{Pcc}
	&P_{c_1,c_2} := \frac{c_1}{p_1} + \frac{c_2}{p_2}
\end{eqnarray}
so that $P = P_{1,1}$.
The main results of this paper are represented in Table~\ref{tab1} and Figure~\ref{fig} showing four sets $R_{11}$, $R_{12}$, $R_{21}$, $R_{22}$ in the parameter region \eqref{pin}  determined by segments $P_{\frac{1}{\alpha}, \frac{1+\alpha}{\alpha}} = 1$, $P_{\frac{1+\alpha}{\alpha}, \frac{1}{\alpha}} = 1 $ with different unbalanced rectangent  limits $V_\pm $.
The critical or the scaling transition point in all four regions $R_{ij}$, $i,j=1,2$, is the same, namely
\begin{eqnarray} \label{gamma0}
	&\gamma_0 := \frac{p_1}{p_2} = \frac{q_1}{q_2}.
\end{eqnarray}
The four RFs $\tilde \Upsilon_{\alpha,i}$, $\Upsilon_{\alpha,i}$, $i=1,2$, in Table~\ref{tab1} are defined in Sec.~\ref{sec4} as integrals with respect to (w.r.t.) $\alpha$-stable random measure on $\R^2$ of non-random integrands determined by certain increments of  $g_0$ in \eqref{glim}
or its partial derivatives. All four RFs  have $\alpha$-stable finite-dimensional distributions and are $(H_1,H_2)$-MSS
with indices shown in Table~\ref{tab1}.
In the Gaussian case $\alpha =2 $ the RFs in Table~\ref{tab1} agree with FBS with the corresponding
parameters $(H_1,H_2)$.

\vskip-.5cm
\begin{figure}[ht]
	\begin{center}
		\includegraphics[width=15 cm,height=16cm, keepaspectratio]{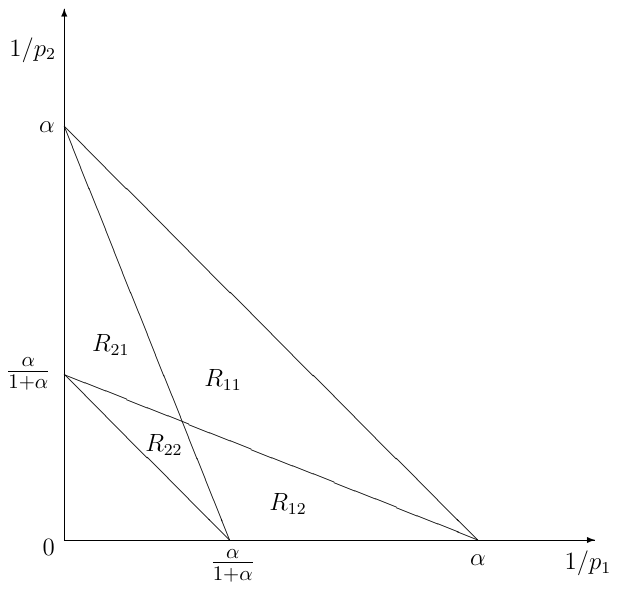}
		\vspace{-9.5cm}
	\end{center}
	\caption{Regions in the parameter set
		$\frac{\alpha}{1+\alpha} < P < \alpha $ with different unbalanced rectangent limits.}\label{fig}
\end{figure}

\vskip-.5cm

\begin{table}[htbp]
	\begin{center}
		\begin{tabular}{ccccc}
			\hline
			{Parameter region}  & {$V_+$} & {Hurst parameters} &  {$ V_-$}  & {Hurst parameters}   \\
			\hline
			{$R_{11}$} & $\tilde \Upsilon_{\alpha,1}$  & $0< H_1 < 1, H_2 =1 $ &  $\tilde \Upsilon_{\alpha,2}$ & $H_1 = 1, 0< H_2 < 1$ \\
			$R_{12}$  & $\tilde \Upsilon_{\alpha,1}$   & $0< H_1 < 1, H_2 =1 $  &   $\Upsilon_{\alpha,1}$ &     $0< H_1 < 1,  H_2 = 0$   \\
			$R_{21}$  &  $\Upsilon_{\alpha,2}$   &   $H_1 = 0, 0< H_2 < 1$  & $\tilde \Upsilon_{\alpha,2}$  &   $H_1=1, 0< H_2 < 1$   \\
			$R_{22}$  &  $\Upsilon_{\alpha,2}$   &   $H_1 =0, 0< H_2 < 1$   &   $\Upsilon_{\alpha,1}$ &   $0< H_1 < 1,  H_2 = 0$     \\
			\hline
		\end{tabular} \\
		\caption{Unbalanced rectangent scaling limits $V_\pm $ and their Hurst parameters in regions $R_{ij}$, $i,j=1,2$, in Figure~\ref{fig}.
		}\label{tab1}
	\end{center}
\end{table}
\vskip-.2cm

Let us briefly describe our results concerning $\gamma$-tangent limits in \eqref{Tlim}.
As mentioned above we consider the case of \eqref{def:X} with $g^0_{1}  = g^0_2 = 0$ so that
\begin{equation}\label{Xinc2}
	X(\mbt) = \int_{\R^2} (g(\mbt-\mbu) - g(-\mbu)) M(\d \mbu), \qquad \mbt \in \R^2,
\end{equation}
provided $g^0_{12} = - g $. When $g(\mbt) = (|t_1|^2 + |t_2|^2)^{\frac{H-1}{2}}$, $\mbt = (t_1,t_2) \in \R^2_0$, is an isotropic  homogeneous function, $0< H< 1$, and $M$ has finite variance, the RF in \eqref{Xinc2} is called \emph{L\'evy fractional RF},  see Sec.~\ref{sec3}.
Tangent ($1$-tangent) RFs of L\'evy fractional RF with truncated $\alpha$-stable $M$ were studied in \cite{benn2004, cohen2012, cohen2013}.
There it was shown that the tangent RF has a similar representation as in  \eqref{Xinc2} with $M$ replaced by $\alpha$-stable $W_\alpha $.
In our paper we extend these results to anisotropic kernels $g$ satisfying \eqref{glim} with $\chi <0$ and more general $M$ and show that they correspond to the well-balanced limit $T_{\gamma_0}$ in \eqref{Tlim} at $\gamma = \gamma_0 $ given in \eqref{gamma0}.
We also prove the existence of the unbalanced limits $T_+ = T_\gamma$ $(\gamma > \gamma_0)$ and $T_- = T_\gamma$ $(\gamma < \gamma_0)$ given by $T_+ (\mbt) := T_{\gamma_0} (t_1) $ and $T_- (\mbt) := T_{\gamma_0} (t_2)$, $\mbt = (t_1,t_2) \in \R^2_+$, which depend on only one coordinate
in contrast to rectangent limits in Table~\ref{tab1}.

The rest of the paper is organized as follows.
Sec.~\ref{sec2} provides rigorous assumptions about $g$ and $M$ and some preliminary facts needed to prove our results.
Sec.~\ref{sec3} presents some examples of fractionally integrated RFs satisfying the assumptions in Sec.~\ref{sec2}.
The $\alpha$-stable MSS RFs in Table~\ref{tab1} are defined in Sec.~\ref{sec4}.
The main results (Theorem \ref{mainthm}) pertaining to Table~\ref{tab1} are given in Sec.~\ref{sec5}.
Sec.~\ref{sec6} discusses $\gamma$-tangent limits.
Some concluding remarks are given in Sec.~\ref{sec7}.

\begin{notation}
	In what follows, $C$ denote  generic positive constants which may be different at different locations.
	We write $\limd$, $\eqd$, $\neqd$ ($\limfdd$, $\eqfdd$, $ \neqfdd$) for the weak convergence, equality, and inequality of (finite-dimensional) distributions.
	$\1 := (1,1)$, $\0 := (0,0)$, $\R^2_0 :=  \R^2 \setminus \{ \0\}$, $\R_0 := \R \setminus \{0\}$,   $\R_+^2 := \{  \mbx = (x_1,x_2) \in \R^2: x_i > 0, \, i=1,2 \}$, $\R_+ := (0,\infty)$, $(\0, \mbx] := (0,x_1]\times (0,x_2]$, $\mbx = (x_1,x_2) \in \R^2_+$, $\mbe_1 = (1,0)$, $\mbe_2 = (0,1)$, $|\mbx| = |x_1| + |x_2|$, $\norm \mbx \norm := (x_1^2 + x_2^2)^{\frac{1}{2}}$, $\mbx \cdot \mby  = x_1 y_1 + x_2 y_2$, $\|f\|_\alpha := (\int_{\R^2} |f(\mbu)|^\alpha \d \mbu)^{\frac{1}{\alpha}}$, $\alpha >0$.
	$I(A)$ stands for indicator function of a set $A$.
\end{notation}

\section{Assumptions and preliminaries}
\label{sec2}

Given a function $f : \R^2 \to \R$ we use the following notation for partial  derivatives at $\boldsymbol{t}=(t_1,t_2) \in \R^2$:
\begin{equation}\label{partialf}
	\partial_i  f(\mbt) = \partial f(\mbt)/\partial t_i, \quad i=1,2, \qquad \partial_{12} f(\mbt)
	:= \partial^2 f(\mbt)/\partial t_1 \partial t_2.
\end{equation}
Following \cite{pils2020}, we say that a measurable function  $f : \R^2_0 \to \R$ is \emph{generalized homogeneous} (respectively, \emph{generalized invariant}) if there exist some positive $q_1, q_2$ such that $\lambda f(\lambda^{1/q_1}t_1, \lambda^{1/q_2} t_2) =  f(\mbt) $ holds for all $\lambda >0$,  $\mbt \in \R^2_0 $ (respectively, $f(\lambda^{1/q_1}t_1, \lambda^{1/q_2} t_2)$ does not depend on $\lambda >0$ for any $ \mbt \in \R^2_0 $).
Every generalized homogeneous function $f(\mbt)$, $\mbt \in \R^2_0$, can be represented as $f(\mbt) = \rho(\mbt)^{-1} \ell (\mbt)$ with $\rho(\mbt) = |t_1|^{q_1} +  |t_2|^{q_2} $ and a generalized invariant function $\ell(\mbt) = \tilde \ell ( t_1/\rho(\mbt)^{1/q_1}, t_2/\rho(\mbt)^{1/q_2})$, where $\tilde \ell$ is a restriction of $f$ to $\{\mbt \in \R^2_0: \rho(\mbt) = 1\}$, see \cite{pils2020}.
We also note that if $f > 0$ is a generalized homogeneous function and $\chi \in \R_0$, then the function $f^\chi $ and its partial derivatives $\partial_i f^\chi$, $i=1,2$, (provided they exist) satisfy the following scaling relations: for all $\lambda >0$, $\mbt \in \R^2_0$,
\begin{equation}\label{fkappa}
	f^\chi (\lambda^{\frac{1}{q_1}} t_1, \lambda^{\frac{1}{q_2}} t_2) = \lambda^{-\kappa} f^\chi(\mbt), \qquad
	\partial_i f^\chi (\lambda^{\frac{1}{q_1}} t_1, \lambda^{\frac{1}{q_2}} t_2) = \lambda^{-\chi -\frac{1}{q_i}} \partial_i f^\chi(\mbt), \quad i=1, 2.
\end{equation}
We shall also need some properties of the above function $\rho$ from  \cite{pils2017, pils2020, sur2019}. Note the elementary inequality: for any $\nu >0$,
\begin{equation}\label{rhoineq}
	C_1 \rho(\mbt)  \le \rho(|t_1|^\nu, |t_2|^\nu)^{\frac 1 \nu} \le C_2 \rho(\mbt), \qquad \mbt \in \R^2,
\end{equation}
with $C_i >0$, $i=1,2$, independent of $\mbt$, see \cite[(2.16)]{sur2019}. From \eqref{rhoineq} and \cite[Prop.~5.1]{pils2017} we obtain for any $\delta, \nu >0$ with $Q = \frac{1}{q_1} + \frac{1}{q_2}$,
\begin{align}\label{intrho}
	\int_{\R^2} \rho (\mbt)^{-\nu}  I( \rho(\mbt) < \delta ) \d \mbt
	< \infty \Longleftrightarrow Q> \nu, \qquad
	\int_{\R^2} \rho (\mbt)^{-\nu}  I( \rho(\mbt) \ge \delta ) \d \mbt
	< \infty \Longleftrightarrow Q< \nu. 
\end{align}
Moreover,  with $q= \max\{ q_1, q_2, 1\}$,
\begin{equation} \label{rhoineq2}
	\rho(\mbt+\mbs)^{\frac 1 q}
	\le \rho(\mbt)^{\frac 1 q} + \rho(\mbs)^{\frac 1 q}, \qquad   \mbt, \mbs \in \R^2,
\end{equation}	
see  \cite[(7.1)]{pils2017}. We shall work on the following assumptions, where $0< \alpha \le 2$ is the stability parameter of limit RFs discussed in the Introduction.

\medskip

\noi {\bf \emph{Assumption} (G)$_\alpha$.} \ The functions $g(\mbt)$ and $g_0(\mbt) = \rho(\mbt)^\chi L(\mbt)$ are as in \eqref{glim}, where $L(\mbt)$ is a generalized invariant function, the parameters $\chi \in \R_0$, $q_i>0$, $i=1,2$, with $Q=\frac{1}{q_1}+\frac{1}{q_2}$ are such that
\begin{eqnarray}\label{qin}
&	- \frac{1}{\alpha} Q< \chi < ( 1-\frac{1}{\alpha} ) Q,
\end{eqnarray}
moreover, $g(\mbt)$ and $g_0(\mbt)$ have partial derivatives in \eqref{partialf} for $\mbt \in \R^2_0$. They satisfy as $|\mbt| \to 0$,
\begin{align}\label{gg0}
	&g(\mbt) = g_0(\mbt)+ o(\rho(\mbt)^\chi), \qquad  \partial_i g(\mbt) = \partial_i g_0(\mbt)+  o(\rho(\mbt)^{\chi - \frac{1}{q_i}} ), \quad i=1,2,\nn\\
	&\qquad\qquad\qquad\partial_{12} g(\mbt) = \partial_{12}g_0(\mbt) +o(\rho(\mbt)^{\chi - Q}),
\end{align}
and for all $\mbt \in \R^2_0$,
\begin{equation}\label{g12bd}
	|g_0 (\mbt)| \le  C\rho(\mbt)^\chi, \qquad |\partial_i g_0 (\mbt)| \le C\rho(\mbt)^{\chi - \frac{1}{q_i}},  \quad i=1,2,  \qquad
	|\partial_{12} g_0(\mbt)| \le C \rho(\mbt)^{\chi - Q}.
\end{equation}

\begin{remark} \label{remkappa}
	In view of \eqref{gg0} the bounds in \eqref{g12bd} extend to $g$ and its derivatives  $\partial_1 g, \partial_2 g, \partial_{12}  g$ in the neighborhood of the origin. We note that the function $g_0(\mbt) = \rho(\mbt)^\chi = (|t_1|^{q_1} + |t_2|^{q_2})^\chi$, corresponding to $L(\mbt) \equiv 1 $ satisfies the bounds in \eqref{g12bd} for all $\boldsymbol{t} \in \R^2_0$ provided $q_i \ge 1$, $i=1,2$. The last fact follows from the expressions of partial derivatives
	\begin{equation*}
		\partial_{i} \rho(\mbt)^\chi = \rho(\mbt)^{\chi - \frac{1}{q_i}} \ell_i(\mbt), \quad i=1,2, \qquad
		\partial_{12} \rho(\mbt)^\chi = \rho(\mbt)^{\chi - Q} \ell_{12}(\mbt)
	\end{equation*}
	with bounded
	\begin{equation*}
		\ell_i(\mbt) := \chi q_i \operatorname{sgn}(t_i) {\textstyle \Big( \frac{|t_i|}{\rho(\mbt)^{\frac{1}{q_i}}} \Big)^{q_i - 1}}, \quad i=1,2, \qquad
		\ell_{12}(\mbt) := \prod_{i=1}^2 (\chi +1-i) q_i \operatorname{sgn}(t_i) {\textstyle \Big( \frac{ |t_i|}{\rho(\mbt)^{\frac{1}{q_i}}} \Big)^{q_i - 1}}.
	\end{equation*}
\end{remark}

\begin{remark}
	Let $\rho_0(\mbt) := |t_1|^{p_1} + |t_2|^{p_2}$, $\mbt \in \R^2$, for $P= \frac{1}{p_1} + \frac{1}{p_2}$ with $p_i$, $i=1,2$, of \eqref{pp}. Using \eqref{rhoineq}, the bounds in \eqref{g12bd} can be respectively replaced by
	\begin{equation}\label{g12bdp}
		|g_0(\mbt)| \le C\rho_0(\mbt)^{P-1}, \qquad |\partial_i g_0(\mbt)| \le
		C\rho_0(\mbt)^{P- \frac{1}{p_i}-1}, \quad i=1,2,   \qquad
		|\partial_{12} g_0(\mbt)| \le C \rho_0(\mbt)^{-1}.
	\end{equation}
\end{remark}

Assumption (G)$_\alpha$ pertains to the behavior of $g$ alone. It is complemented by Assumption (G)$^0_\alpha$ guaranteeing the existence of the associated  RF $X$ in \eqref{def:X}, where infinitely divisible random measure satisfies (M)$_\alpha$.

\medskip

\noi {\bf \emph{Assumption} (G)$^0_\alpha$.} \ For any $\mbt = (t_1,t_2) \in \R^2$, $\delta >0$, the functions $g, g^0_1, g^0_2, g^0_{12} : \R^2 \to \R$ satisfy \begin{equation} \label{G01}
	\int_{\R^2} |g (\mbt - \boldsymbol{u})- g^0_1 (t_1 \mbe_1- \boldsymbol{u})  -  g^0_2 (t_2 \mbe_2 - \boldsymbol{u}) +
	g^0_{12}(-\boldsymbol{u}) |^\alpha \d \mbu <  \infty  \qquad (0 < \alpha \le 2)
\end{equation}
and
\begin{equation} \label{G02}
	\int_{|\mbu| > \delta} \Big( \sum_{i=1}^2 |\partial_i g (\mbu)|^\alpha+ |\partial_{12} g (\mbu)|^\alpha\Big) \d \mbu < \infty
\qquad (1 \le \alpha \le 2).
	\end{equation}
Moreover, if $0 <\alpha < 1$, then
there exist $\delta_0 >0$ and functions $\bar g_i (\mbu)$, $\bar g_{12} (\mbu)$, $\mbu = (u_1,u_2) \in \R^2_+$, 
monotone decreasing in each $u_j >0$, $j=1,2$, and satisfying
$|\partial_i g (\mbu)| \le \bar g_i (|u_1|, |u_2|)$,  $|\partial_{12} g(\mbu)| \le \bar g_{12}(|u_1|,|u_2|)$, $|\mbu| > \delta_0$,
such that \eqref{G02} holds with $\partial_i g$, $\partial_{12}g$ replaced by $\bar g_i, \bar g_{12}$, $i=1,2$.

\medskip

Next we make assumptions about infinitely divisible random measure $M  = \{ M(A), \, A \in {\cal B}_b(\R^2)\}$, where ${\cal B}_b(\R^2)$ denotes the family of all bounded Borel sets.
We recall that the infinitely divisible measure $M$ is such a random process that, for every sequence $\{A_i, \, i \in \N\}$ of pairwise disjoint sets in ${\cal B}_b(\R^2)$, $M(A_i)$, $i \in \N$, are independent random variables (r.v.s), and, if $\cup_{i=1}^\infty A_i \in {\cal B}_b(\R^2)$, then we also have $M(\cup_{i=1}^\infty A_i)  = \sum_{i=1}^\infty M(A_i)$ a.s.
In addition, for every $A \in {\cal B}_b(\R^2)$, $M(A)$ is an infinitely divisible r.v.
We assume that for every $A \in {\cal B}_b(\R^2)$, the characteristic function of $M(A)$ has the form
\begin{equation}\label{def:chfL}
	\E \e^{\i \theta M(A)} =  \exp\Big\{\operatorname{Leb}(A)  \Big(- {\textstyle\frac{1}{2}} \sigma^2 \theta^2 +
	\int_{\R} (\e^{\i \theta y} - 1 - \i \theta \tau_\alpha(y)) \nu (\d y)\Big) \Big\},
	\qquad \theta \in \R,
\end{equation}
where $\sigma^2 \ge 0$, $\nu$ is a L\'evy measure on $\R$ satisfying Assumption (M)$_\alpha$ below and for $y\in\R$,
\begin{equation}
\tau_\alpha(y)
:=
\begin{cases}
y, &1< \alpha \le 2, \\
y \1(|y| \le 1), &\alpha = 1, \\
0,  &0< \alpha < 1.
\end{cases}
\end{equation}
Particularly, when $1< \alpha \le 2 $ we have that
$\int_{\R} |y| \wedge |y|^2  \nu (\d y) < \infty$ and $\E |M(A)| < \infty$, $\E M(A)=0$.

\medskip

\noi {\bf \emph{Assumption} (M)$_\alpha$.} The characteristics $\sigma$, $\nu $ in \eqref{def:chfL} satisfy the following: either
\begin{itemize}
	\item[(i)] $\alpha = 2$, $\sigma >0$ and $\int_{\R} y^2 \nu (\d y) < \infty$, or
\item[(ii)] $ 0< \alpha < 2$, $\sigma =  0 $ and there exist $\lim_{y \downarrow 0} y^\alpha \nu([y,\infty)) = c_+$, $\lim_{y \downarrow 0} y^\alpha \nu((-\infty,-y]) = c_-$ for some $c_\pm \ge 0$, $c_+ + c_- > 0$,  $\sup_{y>0}  y^\alpha \nu(\{u \in \R : |u|>y \}) <  \infty$. Moreover,
if $\alpha =1$, then  $\nu$ is symmetric, i.e.\ $\nu (\d y) = \nu(- \d y)$, $y >0$.
\end{itemize}

\medskip

The above assumption is rather general. For comparison, \cite[Prop.~4.1]{benn2004} consider the truncated stable case $\nu(\d y) = I(|y| \le 1) |y|^{-1-\alpha} \d y $ only, which obviously satisfies (M)$_\alpha$ part (ii) with $c_+ = c_- = \frac 1 \alpha $.
Assumption (M)$_\alpha$ implies that $M$ belongs to  the domain of local attraction of $\alpha$-stable random measure $W_\alpha$ on $\R^2$ with characteristic function
\begin{equation}\label{def:chfW}
	\E \e^{\i \theta W_\alpha(A)} = \e^{-\operatorname{Leb}(A)|\theta|^\alpha \omega_\alpha(\theta)}, \qquad \theta \in \R, \quad A \in
	{\cal B}_b(\R^2),
\end{equation}
where
\begin{equation}
	\omega_\alpha(\theta)
	:=\begin{cases}
		\frac{1}{2}\sigma^2,  &\alpha =2, \\
		\frac{\Gamma(2-\alpha)}{1-\alpha} ((c_+ +  c_-)\cos(\frac{\pi \alpha}{2}) - \i
		(c_+ - c_-)\operatorname{sgn}(\theta) \sin(\frac{\pi \alpha}{2}) ), &0< \alpha < 2, \
\alpha  \neq 1, \\
(c_+ + c_-)\frac{\pi}{2},  &\alpha =1, \ c_+ = c_-.
	\end{cases}
\end{equation}
The last fact can be formulated in terms of local scaling limits of the associated \emph{L\'evy sheet} $\{M(\mbt) := \int_{(\boldsymbol{0}, \mbt]} M(\d \mbu), \, \mbt \in \R^2_+\}$.
Namely, Assumption (M)$_\alpha$ implies that
\begin{equation}\label{LSheet}
	(\lambda_1 \lambda_2)^{-\frac 1 \alpha} M(\lambda_1 t_1, \lambda_2 t_2) \limfdd W_\alpha (\mbt), \qquad \lambda_i \downarrow 0, \quad i=1,2,
\end{equation}
where  $\{W_\alpha(\mbt) := \int_{(\boldsymbol{0}, \mbt]} W_\alpha(\d \mbu), \, \mbt \in \R^2_+\}$ is $\alpha$-stable L\'evy sheet defined by \eqref{def:chfW}.
The fact in \eqref{LSheet} seems to be well-known, see e.g.\ \cite[Prop.~5.1]{bard2011} and also follows from Proposition~\ref{offp} providing a general criterion for weak  convergence of stochastic integrals
w.r.t.\ infinitely divisible random measure $M$ towards that w.r.t.\ $\alpha$-stable random measure $W_\alpha$.
Recall that stochastic integral  $\int_{\R^2} f(\mbu) W_\alpha(\d \mbu) $ is well-defined
for any $f \in L_\alpha(\R^2)$  and has $\alpha$-stable distribution and characteristic function \cite{samo1994}
\begin{equation}\label{chfIW}
	\E \exp\Big\{ \i \theta \int_{\R^2} f(\mbu) W_\alpha(\d \mbu)\Big\} =
	\exp\Big\{ - |\theta|^\alpha \int_{\R^2} |f(\mbu)|^{\alpha} \omega_\alpha(\theta f(\mbu)) \d \mbu \Big\}, \qquad \theta \in \R.
\end{equation}
Also recall that stochastic integral $\int_{\R^2} f(\mbu) M(\d \mbu) $ w.r.t.\ infinitely divisible random measure satisfying Assumption (M)$_\alpha$  is well-defined for any $f \in L_\alpha(\R^2)$ and has an infinitely divisible distribution with characteristic function
\begin{align}\label{chfIM}
\E &\exp \Big\{ \i \theta \int_{\R^2} f(\mbu)  M(\d \mbu) \Big\} \nn\\
		= &\exp \Big\{ \int_{\R^2} \Big(-{\textstyle \frac{1}{2}}\sigma^2 \theta^2 f(\mbu)^2 +
\int_{\R} (\e^{\i \theta f(\mbu)y} - 1 - \i \theta f(\mbu) \tau_\alpha (y)) \nu (\d y) \Big) \d \mbu \Big\}.
\end{align}
To prove the (weak) convergence of stochastic integrals in \eqref{chfIM} towards stable integral in  \eqref{chfIW},
we  use the following proposition.  For any $0  < \alpha \le 2$, $\mu_i  >0$, $i=1,2$, and any function $f = f_\lambda: \R^2 \to \R $  possibly depending on $\lambda >0$, define the re-scaled function $f^\dagger_{\lambda}: \R^{2} \to \R $ by
\begin{equation}\label{tildeg}
	f^\dagger_{\lambda}(\mbu):=\lambda^{\frac{1}{\alpha}(\mu_1+\mu_2)}
	f_\lambda(\lambda^{\mu_1}  u_1, \lambda^{\mu_2} u_2), \qquad \mbu \in \R^{2}.
\end{equation}

\begin{proposition} \label{offp}
	Assume that the infinitely divisible random measure $M$ satisfies (M)$_\alpha$ for some $0 < \alpha \le 2 $. Let $f_{\lambda} \in L_{\alpha}(\R^{2})$, $\lambda >0$. If there exists $h \in L_{\alpha} (\R^{2})$ such that
	\begin{equation}\label{condf}
		\lim_{\lambda \downarrow 0} \|f^\dagger_{\lambda} - h\|_{\alpha} = 0,
	\end{equation}
	for some $\mu_i >0$, $i=1,2$, then
	\begin{equation}\label{Iconvp}
		\int_{\R^2} f_{\lambda}  (\mbu) M(\d \mbu) \limd
\int_{\R^{2}} h(\mbu) W_\alpha(\d \mbu), \qquad \lambda \downarrow 0.
	\end{equation}
\end{proposition}

\begin{proof}
	It suffices to prove the convergence of characteristic functions: $ \E \exp \{\i \theta\int_{\R^2} f_{\lambda}  (\mbu) M(\d \mbu) \}=:C_\lambda(\theta) \to C_0(\theta) := \E \exp \{\i \theta\int_{\R^2} h (\mbu) W_\alpha (\d \mbu) \}$, $\forall \theta \in \R$, as given in \eqref{chfIM}, \eqref{chfIW}.
	W.l.g., assume that $f^\dagger_{\lambda}(\mbu)  \to  h(\mbu)$ for a.e.\ $\mbu \in \R^2$.
	Split $M(\d \mbu) = M_1(\d \mbu) + M_2(\d \mbu)$, where $M_1$ (respectively, $M_2$) is infinitely divisible random measure with characteristics $(0, \nu)$ (respectively, $(\sigma, 0)$) and $M_1$, $M_2$ are independent.
	Accordingly, $C_\lambda(\theta) = C_{\lambda,1} (\theta)C_{\lambda,2}(\theta)$, where $C_{\lambda,i}(\theta) :=  \E \exp \{\i \theta\int_{\R^2} f_{\lambda}  (\mbu) M_i(\d \mbu) \}$, $i=1,2$. The subsequent proof is split into four parts depending on the value of $\alpha$. Let
\begin{equation*}
\tilde \lambda := \lambda^{\mu_1 + \mu_2}, \qquad \Psi_\alpha (y) := \e^{\i y} - 1 - \i  \tau_\alpha(y) \quad (0 < \alpha \le 2, \ y \in \R).
\end{equation*}

\noi Case $\alpha =2$. Then 
$\sigma >0$ and
$C_{\lambda,2}(\theta) \to C_0(\theta) $ is immediate from \eqref{condf}, hence, \eqref{Iconvp} follows from $C_{\lambda,1}(\theta) \to 1$.
We have $C_{\lambda,1} (\theta) = \e^{I_\lambda}$, where $I_\lambda = \int_{\R^2\times \R} \Psi_2 (\theta f_\lambda (\mbu) y) \d \mbu \nu (\d y) = \tilde \lambda \int_{\R^2\times \R} \Psi_2 (\theta \tilde \lambda^{-\frac 1 2} f^\dagger_\lambda (\mbu) y ) \d \mbu \nu (\d y)$. By Pratt's lemma \cite{prat1960}, $I_\lambda= o(1)$. Indeed, using $|\Psi_2 (y)| \le \min\{2|y|, \frac{1}{2}|y|^2\}$ gives
	$\tilde \lambda |\Psi_2 (\theta \tilde \lambda^{-\frac 1 2} f^\dagger_\lambda (\mbu) y ) |
	\le C \tilde \lambda^{\frac 1 2} |f^\dagger_\lambda (\mbu) y| = o(1)$ for $y \in \R $, a.e.\ $\mbu \in \R^2 $ such that $f^\dagger_\lambda (\mbu) \to h(\mbu)$, moreover, $\tilde \lambda |\Psi_2 (\theta \tilde \lambda^{-\frac 1 2} f^\dagger_\lambda (\mbu) y ) | \le C |f^\dagger_\lambda (\mbu) y|^2$  with $\int_{\R^2\times \R} |f^\dagger_\lambda (\mbu) y|^2 \d \mbu \nu (\d y) \to \int_{\R^2\times \R} |h(\mbu) y|^2 \d \mbu \nu(\d y)$, proving \eqref{Iconvp} for $\alpha = 2$.

\smallskip

\noi Case $1< \alpha <2$. 	
		In this case $\sigma  = 0$.
Let $\bar \nu_+(y) := \nu ([y, \infty))$, $\bar \nu_-(y) := \nu((-\infty, -y])$, $y >0$.
Then $C_\lambda (\theta) = \exp \{\tilde \lambda \int_{\R^2\times \R} \Psi_\alpha (\theta \tilde \lambda^{- \frac 1 \alpha} f^\dagger_\lambda (\mbu) y ) \d \mbu \nu (\d y) \} =: \e^{I_\lambda}$, where $I_\lambda = \int_{\R^2 \times \R} \Phi_\lambda (\theta f^\dagger_\lambda (\mbu), y) \d \mbu \d y$  follows integrating by parts with
\begin{eqnarray*}
&\Phi_\lambda(u, y) := \tilde \lambda \bar \nu_{\operatorname{sgn}(y)}
\Big( \frac{ {\tilde \lambda}^{\frac{1}{\alpha}} |y|} {| u |} \Big)
(\e^{\i \operatorname{sgn}(u) y} -1 ) \i \operatorname{sgn}(u y)
\\
&\to |u|^\alpha  \frac{  c_{\operatorname{sgn} (y)} }{|y|^\alpha} (\e^{\i \operatorname{sgn}(u)  y} -1 ) \i  \operatorname{sgn}(u y) =: \Phi_0(u,y),\nn
\end{eqnarray*}
hence, $\Phi_\lambda(\theta f^\dagger_\lambda(\mbu), y) \to \Phi_0 (\theta h(\mbu), y)$ for any $(\mbu, y) \in \R^2 \times \R $ such that $f^\dagger_\lambda(\mbu) \to h(\mbu)$.
Therefore we can expect that
\begin{equation}\label{Jconv}
I_\lambda \to \int_{\R^2 \times \R} \Phi_0 (\theta h(\mbu), y) \d \mbu \d y
=  - |\theta|^\alpha \int_{\R^2} |h(\mbu)|^\alpha \omega_\alpha (\theta h(\mbu)) \d \mbu,
\end{equation}
where the last equality follows from \cite[proof of Thm.~2.2.2]{ibra1971}.
To justify convergence of the integrals in \eqref{Jconv}, by $|\bar \nu_\pm (y)| \le C y^{-\alpha}$, $y >0$, we have $|\Phi_\lambda(u, y)| \le C |u|^\alpha  |y|^{-\alpha} \min\{|y|,1\} =: \bar \Phi(u,y)$, implying $\Phi_\lambda(\theta f^\dagger_\lambda (\mbu), y) \le \bar \Phi(\theta f^\dagger_\lambda (\mbu),y)$, where $\bar \Phi(\theta f^\dagger_\lambda (\mbu),y) \to \bar \Phi(\theta h (\mbu),y)$ for a.e.\ $(\mbu, y) \in \R^2 \times \R $ and $\int_{\R^2 \times \R} \bar \Phi(\theta f^\dagger_\lambda (\mbu)) \d \mbu \d y \to \int_{\R^2 \times \R} \bar \Phi(\theta h(\mbu),y) \d \mbu \d y < \infty$ so that \eqref{Jconv} follows by Pratt's  lemma as in the case $\alpha =2 $ above.
Hence, $C_\lambda(\theta) \to C_0(\theta)$ $(\forall \theta \in \R)$.

\smallskip
	
\noi Case $0<\alpha < 1$. Then
		$C_\lambda (\theta) =  \exp\{ \int_{\R^2} J_\lambda (\theta f^\dagger_\lambda (\mbu)) \d \mbu\}$, where
		for a large $K>0$ we split $J_\lambda (u)= J_{\lambda,K,0} (u)+ J_{\lambda,K,1}(u)$ with
		\begin{equation*}\label{Jsplit}
		J_{\lambda,K,0} (u) :=
		\tilde \lambda
		\int_{|y| \le \frac{\tilde \lambda^{\frac{1}{\alpha}}K}{|u|}}
		(\e^{\i \tilde \lambda^{-\frac{1}{\alpha}} uy} -1) \nu (\d y),\quad
		J_{\lambda,K,1} (u) :=
		\tilde \lambda
		\int_{|y| > \frac{\tilde \lambda^{\frac{1}{\alpha}}K}{|u|}}
		(\e^{\i \tilde \lambda^{-\frac{1}{\alpha}} uy} -1) \nu (\d y).
		\end{equation*}		
		Integrating by parts
we rewrite $J_{\lambda,K,0} (u) =  \Phi_\lambda (u,K) - \Phi_\lambda (u,-K) - \int_{|y| \le K} \phi_\lambda (u,y) \d y$ with
		\begin{eqnarray*}
		&\Phi_\lambda (u,y) :=
		- \tilde \lambda \bar \nu_{\operatorname{sgn}(y)} \Big(\frac{\tilde \lambda^{\frac{1}{\alpha}} |y|}{|u|} \Big)
		\operatorname{sgn}(y) (\e^{\i \operatorname{sgn}(u) y}- 1),\\
		&\phi_\lambda (u,y) :=
		- \tilde \lambda \bar \nu_{\operatorname{sgn}(y)} \Big(\frac{\tilde \lambda^{\frac{1}{\alpha}} |y|}{|u|} \Big)
		\i \operatorname{sgn}(u y) \e^{\i \operatorname{sgn}(u) y}.
		\end{eqnarray*}			
		In the corresponding decomposition of $I_{\lambda,K,0} := \int_{\R^2} J_{\lambda,K,0}(\theta f^\dagger_\lambda(\mbu)) \d \mbu$ we use  $\Phi_\lambda (\theta f^\dagger_\lambda (\mbu),y) \to \Phi_0 (\theta h(\mbu),y)$ and $\phi_\lambda (\theta f^\dagger_\lambda (\mbu),y) \to \phi_0 (\theta h(\mbu),y)$ for any $(\mbu,y) \in \R^2 \times \R$ such that $f^\dagger_\lambda (\mbu) \to h(\mbu)$, where
		\begin{eqnarray*}
		\Phi_0 (u,y) &:= - |u|^\alpha \frac{c_{\operatorname{sgn}(y)}}{|y|^\alpha} \operatorname{sgn}(y) (\e^{\i \operatorname{sgn}(u) y}- 1),\qquad
		\phi_0 (u,y) := - |u|^\alpha \frac{c_{\operatorname{sgn}(y)}}{|y|^\alpha}  \i \operatorname{sgn}(uy) \e^{\i \operatorname{sgn}(u) y}.
		\end{eqnarray*}
		Arguing as in the proof of \eqref{Jconv}, we obtain $I_{\lambda,K,0} \to I_{K,0} := \int_{\R^2} J_{K,0} (\theta h(\mbu)) \d \mbu$ with
		\begin{equation*}
		J_{K,0} (u) := \Phi_0(u,K)- \Phi_0(u,-K) - \int_{|y|\le K} \phi_0(u,y) \d y
		= |u|^\alpha \int_{|y| \le K} {\textstyle\frac{\alpha c_{\operatorname{sgn}(y)}}{|y|^{\alpha+1}} (\e^{\i \operatorname{sgn} (u) y} - 1)} \d y,
		\end{equation*}
		where $I_{K,0} \to - |\theta|^\alpha \int_{\R^2} |h(\mbu)|^\alpha \omega_\alpha (\theta h(\mbu)) \d \mbu$ as $K \to \infty$. 		
		The fact that $\int_{\R^2} J_{\lambda,K,1}(\theta f^\dagger_\lambda (\mbu)) \d \mbu =: I_{\lambda,K,1}$ is negligible follows from
		$|I_{\lambda,K,1}| \le  2 \tilde \lambda \int_{\R^2} \nu ( |y| > \frac{\tilde \lambda^{\frac{1}{\alpha}} K} {|\theta f^\dagger_\lambda(\mbu)|} ) \d \mbu
		\le \frac{C}{K^{\alpha}} \int_{\R^2} |\theta f^\dagger_\lambda(\mbu)|^\alpha \d \mbu
		\le \frac{C}{K^{\alpha}}$, by choosing $K$ large enough.
		This proves $C_\lambda(\theta) \to C_0(\theta)$ $(\forall \theta \in \R)$ in case $0< \alpha < 1 $.

\smallskip

\noi Case $\alpha =1$.  By symmetry of $\nu$ and $W_1$,
$ C_\lambda (\theta) = \e^{2I_\lambda}, C_0(\theta)
= \e^{2 I_0}$, where
\begin{align*}
&I_\lambda := \tilde \lambda \int_{\R^2 \times \R_+} (\cos(\theta f^\dagger_\lambda(\mbu) \tilde \lambda^{-1} y) - 1) \d \mbu \nu (\d y), \\
&I_0 :=  c_+ \int_{\R^2 \times \R_+} (\cos(\theta h(\mbu) y) - 1) y^{-2} \d \mbu \d y = - c_+ |\theta| \| h \|_1 J,
\end{align*}
where $J := \int_0^\infty (1-\cos y) y^{-2} \d y = \int_0^\infty \frac{\sin y}{y} \d y = \frac{\pi}{2}$. 
Similarly as in the case $0<\alpha < 1$, for a large $K>0$ split $I_\lambda = I_{\lambda,K,0} + I_{\lambda,K,1}$, where
\begin{equation*}
I_{\lambda,K,0}
:= \tilde \lambda \int_{\R^2 \times \big\{0< y \le \frac{\tilde \lambda K}{|\theta f^\dagger_\lambda(\mbu)|}\big\}} \dots, \qquad
I_{\lambda,K,1} := \tilde \lambda \int_{\R^2 \times \big\{y> \frac{\tilde \lambda K}{|\theta f^\dagger_\lambda(\mbu)|}\big\}} \dots.
\end{equation*}
Then integrating by parts and using (M)$_1$, Pratt's lemma similarly as in the case $0<  \alpha < 1 $ we obtain
\begin{equation*}
I_{\lambda,K,0}
\to c_+\int_{\R^2} |\theta h(\mbu) | \d \mbu \int_0^K (\cos y -1)
y^{-2} \d y  =: I_{K,0},
\end{equation*}
where $I_{K,0} \to I_0$ $(K \to \infty)$ while $I_{\lambda,K,1}$ can be made arbitrarily small uniformly in $\lambda >0$ by choosing $K  >0$ large enough.
Proposition \ref{offp} is proved.
\end{proof}

\begin{remark} \label{rem3}
	We have $\lambda^{-\frac{1}{\alpha}(1+ \gamma)} M(\lambda t_1, \lambda^\gamma t_2)= \int_{\R^2} f_\lambda (\mbu) M(\d \mbu) $ with $f_\lambda (\mbu) :=  \lambda^{-\frac{1}{\alpha}(1+ \gamma)}  I(\mbu \in (0, \lambda t_1] \times (0, \lambda^\gamma t_2])$.
	Then if $\mu_1 = 1$, $\mu_2 = \gamma $ the corresponding re-scaled function in  \eqref{tildeg} $f^\dagger_\lambda (\mbu) =   I(\mbu \in (\0, \mbt]) $ does not depend on $\lambda $ and trivially satisfies \eqref{condf} with $h (\mbu) =   I(\mbu \in (\0, \mbt])$.
	Accordingly by Proposition  \ref{offp}  the convergence in \eqref{LSheet} holds when $\lambda_i$, $i=1,2$, tend to zero as  $\lambda_1 = \lambda$, $\lambda_2 = \lambda^\gamma$.
	Actually Proposition  \ref{offp} extends to more general scaling by $\lambda_i \downarrow 0$, $i=1,2$, as in \eqref{LSheet} but for our purposes the scaling in \eqref{tildeg} suffices.
\end{remark}

\section{Examples of fractional L\'evy driven RFs}
\label{sec3}

In this section we discuss three examples of L\'evy driven fractional RFs related to fractional powers of classical partial differential operators. Fractional operators and equations naturally appear
in the study of fractional RFs with LRD or negative dependence,  
see  \cite{kelb2005, leo2011} and references therein.
In each of these examples Assumptions (G)$_\alpha$ and (G)$^0_\alpha$ are verified and Theorems   \ref{mainthm} and \ref{mainthm2} apply in respective parameter regions.

\medskip

\noi {\bf 3.1 \ Fractional L\'evy RF.}
\ Let Assumption (M)$_\alpha$ hold and $g(\boldsymbol{t}) = \norm \boldsymbol{t}\norm^{H-\frac{2}{\alpha}}$, $g^0_i(\mbt) = 0$, $i=1,2$, $g^0_{12}(\mbt) = - \norm \boldsymbol{t}\norm^{H-\frac{2}{\alpha}}$ for all $\mbt \in \R^2_0$ and some $H \in (0,1)$, $\alpha \in (0,2]$.
The corresponding L\'evy driven  RF $X$ takes the form
\begin{equation}\label{XLevy}
	X(\boldsymbol{t}) :=  \int_{\R^2} \{ \norm \mbt- \mbu \norm^{H-\frac{2}{\alpha}} -  \norm \mbu \norm^{H-\frac{2}{\alpha}} \}
	M(\d \boldsymbol{u}),  \qquad \boldsymbol{t}  \in \R^2.
\end{equation}
If $\E M(\d \mbu)^2 = \sigma^2_M \d \mbu$ and $\alpha=2$ in the integrand, then the covariance function of $X$ in \eqref{XLevy} is given by
\begin{eqnarray}
	&\E X(\mbt) X(\mbs) = \E |X(\mbe_1)|^2 \frac{1}{2} (\norm \mbt \norm^{2H} + \norm \mbs \norm^{2H} - \norm \mbt - \mbs \norm^{2H}), \qquad \mbt, \mbs \in \R^2,
\end{eqnarray}		
where $\E |X(\mbe_1)|^2 = \sigma_M^2 \int_{\R^2} (\norm \mbe_1 + \mbu \norm^{H-1} - \norm \mbu \norm^{H-1})^2 \d \mbu < \infty$.
If $M  = W_\alpha$ is $\alpha$-stable random measure, then RF $X$ in \eqref{XLevy} is called a moving average fractional stable RF (a fractional Brownian RF in the case $\alpha=2$), see \cite{benn2004, cohen2012, take1991}.
Fractional L\'evy RF can be defined on arbitrary $\R^d$, $d \ge 1$.
See the review paper \cite{lodh2016} and the numeruous references therein on various mathematical and probabilistic aspects of the fractional Brownian RF, including extension to arbitrary $H \in \R$ (as a generalized RF or random tempered distribution) and relation to fractional powers of the Laplace operator.
Remark \ref{remkappa} shows that the above $g$ verifies Assumption (G)$_\alpha$ with $g_0 = g$, $q_1 = q_2 =2$, $Q=1$, $\chi = \frac{H}{2}- \frac{1}{\alpha} \in (-\frac{1}{\alpha},\frac{1}{2}- \frac{1}{\alpha}).$
Moreover, parameters in \eqref{pp} of the limiting kernel satisfy: $p_1 = p_2 = 2-H+\frac{2}{\alpha}$ with $\frac{\alpha}{1+\alpha} < P < \frac{2\alpha}{2+\alpha}$, $P_{\frac{1+\alpha}{\alpha},\frac{1}{\alpha}} = P_{\frac{1}{\alpha}, \frac{1+\alpha}{\alpha}} <1$.

\medskip

\noi {\bf 3.2 \ Isotropic fractional Laplace or Mat\'ern RF.}
\ Let
\begin{eqnarray}\label{gEx2mod}
&g(\boldsymbol{t}) := \frac{2^{1+\chi}}{c^{2\chi} \Gamma (-\chi)}
\{
(c \norm \boldsymbol{t} \norm)^{\chi} K_{\chi} (c\norm \boldsymbol{t}\norm ) - I (\chi >0) \Gamma(\chi) 2^{\chi-1}
\},
\qquad \mbt \in \R^2_0,
\end{eqnarray}
where $c >0$, $\chi \in ({\textstyle -\frac{1}{\alpha}},1-{\textstyle\frac{1}{\alpha}}) \setminus \{0\}$ and $K_{\chi}$ denotes the modified Bessel function of the second kind.
Using $K_{\chi} (t) = K_{-\chi} (t) \sim \Gamma(\chi) 2^{\chi-1} t^{-\chi}$ as $t \downarrow 0$ ($\chi>0$), see \cite[9.6.9, p.375]{abra1972}, we see that $g(\mbt) \sim \norm \mbt\norm^{2\chi}$ $(\chi <0)$, $g(\mbt) \to 0$ $(\chi >0)$ as $\norm \mbt \norm \to 0$; moreover, for $2 \chi = H-\frac{2}{\alpha} <0$ we have $\lim_{c \downarrow 0} g(\mbt) = \norm \mbt \norm^{H-\frac{2}{\alpha}} $ hence $g(\mbt)$ of Example 3.1 can be regarded as the limiting case of \eqref{gEx2mod} when $c \downarrow 0$.
Let Assumption (M)$_\alpha$ hold for $\alpha \in (0, 2]$ and
\begin{equation}\label{XLevy2}
X(\mbt) := \int_{\R^2} \{ g(\mbt - \mbu) - I (\chi > 0) g(-\mbu) \} M(\d \mbu), \qquad \mbt \in \R^2.
\end{equation}
Clearly \eqref{XLevy2} is a particular case of \eqref{def:X} corresponding to $g^0_1(\mbt) = g^0_2(\mbt) = g^0_{12}(\mbt) = 0$ $(\chi <0)$ and $g^0_1(\mbt) = g^0_2(\mbt) = 0$, $g^0_{12}(\mbt) = - g(\mbt)$ $(\chi >0)$.

\begin{proposition}
	For any $\alpha \in (0,2]$ the kernel $g(\mbt) $ in \eqref{gEx2mod} satisfies Assumption (G)$_\alpha$ with
	\begin{equation}\label{Ex2g0}
		g_0(\mbt) := \norm \mbt \norm^{2\chi}, \qquad \mbt \in \R^2_0,
	\end{equation}
	$q_1 = q_2 = 2, Q=1, $ and $\chi $ as in \eqref{gEx2mod}, or $p_1 = p_2 = 2(1-\chi)$, $P \in (\frac{\alpha}{1 + \alpha}, \alpha)$, $P \neq 1$.
	Moreover, the integrand in  \eqref{XLevy2} satisfies Assumption (G)$^0_\alpha$.
	As a consequence, the RF $X$ is  \eqref{XLevy2} is well-defined for any $M$ satisfying Assumption  (M)$_\alpha$.
\end{proposition}

\begin{proof}
	Using the relation $K_\chi (t) = K_{-\chi} (t) = \frac{1}{2} \Gamma(\chi) \Gamma(1-\chi)(I_{-\chi}(t) - I_\chi(t))$ (c.f.\ \cite[9.6.2, 9.6.6, p.375]{abra1972}), where the modified Bessel function of the first kind can be expressed as $I_\nu (t) = \frac{1}{2^{\nu}\Gamma (1+\nu)} t^\nu + \frac{1}{2^{2+\nu}\Gamma (2+\nu)} t^{2+\nu} + o (t^{2+\nu})$, $t \downarrow 0$, (c.f.\ \cite[9.6.10, p.375]{abra1972}), the asymptotics $g(\mbt) = g_0(\mbt) (1+o(1))$ with $g_0(\mbt)$ as in \eqref{Ex2g0} follows
	since $|\chi|<1$ in case $0< \chi < 1-\frac{1}{\alpha}$, and as shown before in case $\chi < 0$. Note the derivatives
	\begin{equation*}\label{partg}
	\partial_i g_0(\mbt) = 2\chi \norm \mbt \norm^{2\chi-1} \partial_i \norm \mbt \norm, \quad i=1,2, \qquad
	\partial_{12} g_0(\mbt) = 4 \chi (\chi -1) \norm \mbt \norm^{2\chi-2} \prod_{i=1}^2 \partial_i \norm \mbt \norm.
	\end{equation*}
	It suffices to verify \eqref{gg0} for derivatives $\partial_{12} \tilde g (\mbt)$, $\partial_i \tilde g (\mbt)$, $i=1,2$, where $\tilde g(\mbt) := \frac{2^{1+\chi}}{\Gamma (-\chi)} \norm t \norm^{\chi} K_{\chi} (\norm t \norm)$.
	Using the recurrence relation $K'_\nu (t) = - K_{\nu-1} (t) - \nu t^{-1} K_\nu (t)$ (c.f.\ \cite[9.6.26, p.376]{abra1972}) and $\partial_i \norm \mbt \norm = \norm \mbt \norm^{-1} t_i$, we find
	\begin{eqnarray*}
	&\partial_{i} \tilde g (\mbt) = \frac{2^{1+\chi}}{\Gamma (-\chi)} (\chi \norm \mbt \norm^{\chi -1} K_\chi (\norm \mbt \norm)  + \norm \mbt \norm^\chi K'_\chi (\mbt) ) \partial_i \norm \mbt \norm = - \frac{2^{1+\chi}}{\Gamma (-\chi)} \norm \mbt \norm^\chi K_{\chi-1} (\norm \mbt \norm) \partial_i \norm \mbt \norm,  \quad i=1,2,
	\end{eqnarray*}
	and
	\begin{eqnarray*}
	&\partial_{12} \tilde g(\mbt) = - \frac{2^{1+\chi}}{\Gamma (-\chi)} ( (\chi-1) \norm \mbt \norm^{\chi-2} K_{\chi-1} (\norm \mbt \norm) + \norm \mbt \norm^{\chi-1} K'_{\chi-1} (\norm \mbt \norm) ) t_1 \partial_2 \norm \mbt \norm\\
	&= \frac{2^{1+\chi}}{\Gamma(-\chi)}\norm \mbt \norm^\chi K_{\chi-2}(\norm \mbt \norm) \prod_{i=1}^2 \partial_i \norm \mbt \norm.
	\end{eqnarray*}
	Thus \eqref{gg0} follows using $ K_{-\nu} (t) = K_\nu (t)  \sim \Gamma (\nu) 2^{\nu-1} t^{-\nu}$, $t \downarrow 0$ $(\nu >0)$.
	The remaining facts of Assumption (G)$_\alpha$ follow from the definition in \eqref{gEx2mod}.
	Finally, Assumption (G)$^0_\alpha$ is guaranteed since $K_\nu (t)$ decays exponentially as $t \to \infty$ for any $\nu >0$.
\end{proof}

\begin{remark} The stationary Mat\'ern RF on $\R^2 $ is defined by
	\begin{equation} \label{YLevy2}
		Y(\mbt) := \int_{\R^2} h(\mbt - \mbu) M(\d \mbu), \qquad \mbt \in \R^2,
	\end{equation}
	where
	\begin{eqnarray*}
	&	h(\mbt) := \frac{2^{1+\chi}}{c^{\chi}\Gamma (-\chi)}
		\norm \boldsymbol{t} \norm^{\chi} K_{\chi} (c\norm \boldsymbol{t}\norm), \qquad \chi > - \frac{1}{\alpha},
	\end{eqnarray*}
	agrees with $g(\mbt) $ in \eqref{gEx2mod} for $\chi <0$; for $\chi >0$ we have  $g(\mbt) = h(\mbt) - \lim_{\norm \mbs \norm \to 0} h(\mbs)$. Therefore,
	for $\chi >0$ the RF
	$X  $ in \eqref{XLevy2} is the increment RF  $X(\mbt) = Y(\mbt) - Y(\0)$, $\mbt \in \R^2$, of the Mat\'ern RF $Y$ in \eqref{YLevy2}. Clearly, RF
	$X$ and $Y$ have identical ordinary and rectangular increments and tangent limits.  Finite-variance Mat\'ern RFs and their covariance functions
	are widely used in spatial applications, see \cite{gutt2006, jons2013, hansen2013, bol2014, wal2015} and the references therein. If $\E M(\d \mbu)^2=\sigma_M^2 \d \mbu$ and $\chi>-\frac{1}{2}$, then $\E |Y(\0)|^2<\infty$ and the covariance function of RF $Y$
	is given by
	\begin{eqnarray*}
		&R(\mbt) := \E Y(\0) Y(\mbt) = \E |Y (\0)|^2 \frac{(c \norm \boldsymbol{t} \norm )^{1 +2  \chi} K_{1+ 2\chi} (c \norm
			\boldsymbol{t} \norm)}{\Gamma (1+ 2\chi) 2^{2\chi}},
		\qquad \boldsymbol{t} \in \R^2.
	\end{eqnarray*}
	Whence, $\E X(\mbt) X(\mbs) = R(\mbt - \mbs) - R(-\mbs) - R(\mbt) + R(\0)$  and $\E (X(\mbt) - X(\mbs))^2
	= 2(R(\0) - R(\mbt-\mbs))$,
	$\mbt, \mbs \in \R^2$, for $\chi \in (0,\frac{1}{2})$. We note that
	the finite-variance RF $Y$ in \eqref{YLevy2} can be regarded as a stationary solution
	of the fractional Helmholtz equation
	\begin{equation*}
	(c^2 - \Delta)^{1+ \chi} Y (\mbt) =  \dot M(\mbt), \qquad \mbt \in \R^2,
	\end{equation*}
	where $\Delta  := \frac{\partial^2}{\partial t_1^2} +  \frac{\partial^2}{\partial t_2^2} $ is the Laplace operator
	and $\dot M (\mbt)= \partial_{12} M (\mbt)$ is the L\'evy white noise (the generalized random process),
	see  \cite{whit1963, bol2014}. Non-Gaussian Mat\'ern RF with
	$M$ belonging to some parametric class are discussed in \cite{bol2014,wal2015}.
\end{remark}

\medskip

\noi {\bf 3.3 \ Anisotropic fractional heat operator RF. } \ Let
\begin{eqnarray}\label{ex3g}
	&g(\mbt) =  \frac{t_1^\chi} {2^{\frac{1}{2}}
		(2\pi)^{\frac{3}{2}}  c_2  \Gamma(\chi + \frac{3}{2})}  \exp \{-c_1  t_1 - \frac{t_2^2}{4 c_2^2 t_1}\} I(t_1 >0),
	\qquad \mbt \in \R^2,
\end{eqnarray}
where $\chi > - \frac{3}{2 \alpha}$, $c_1 >0$, $c_2 >0$ are parameters.
In the case $\alpha=2$, the kernel in \eqref{ex3g} is related to the fractional heat operator  $(c_1 + \Delta_{12})^{\chi + \frac{3}{2}}$, $\Delta_{12}  := \frac{\partial}{\partial t_1} - c_2^2 \frac{\partial^2}{\partial t_2^2} $, as explained below.
(For $\chi = -\frac{1}{2} $ it solves the equation $(c_1 + \Delta_{12}) g(\mbt) = 0$, $t_1 >0$, as expected.)
The stationary solution of the corresponding stochastic equation
\begin{equation}\label{Ex3}
(c_1 + \Delta_{12})^{\chi + \frac{3}{2}} X (\mbt) =  \dot M(\mbt), \qquad \mbt \in \R^2,
\end{equation}
with Gaussian white noise $\dot M$ is defined in \cite[(3.2)]{kelb2005} as a moving average RF
\begin{equation} \label{XLevy3}
	X(\mbt) = \int_{\R^2} g(\mbt - \mbu) M(\d \mbu)
\end{equation}
with the spectral density (given by the squared Fourier transform of  $g$) of the form
\begin{eqnarray}\label{ex3f}
&	f(\mbx) = \frac{\sigma_M^2}{(2\pi)^2} |\widehat g(\mbx)|^2 =
	\frac{\sigma_M^2}{(2\pi)^2} \frac{1}{(x_1^2 + (c_1 + c^2_2 x_2^2)^2)^{\chi + \frac{3}{2}}}, \qquad \mbx \in \R^2.
\end{eqnarray}
We claim that the corresponding $g \in L_2 (\R^2)$ is given by \eqref{ex3g} for $\chi > - \frac{3}{4}$. Indeed, its Fourier transform  can be found from \cite[3.944.5--6]{grad2000}:
\begin{align}
	\int_{\R^2} g(\mbt) \e^{\i \mbt \cdot \mbx} \d \mbt
	&= {\textstyle \frac{1}{2\pi \Gamma(\chi+\frac{3}{2})}} \int_0^\infty t_1^{\chi} \e^{-c_1 t_1 + \i x_1 t_1} \d t_1
	\int_{\R} {\textstyle\frac{1}{(4\pi c_2^2 t_1)^{\frac{1}{2}}} \e^{-  \frac{t_2^2}{4 c_2^2 t_1} + \i x_2 t_2 }} \d x_2 \nn \\
	&= {\textstyle\frac{1}{2\pi \Gamma(\chi+\frac{3}{2})}} \int_0^\infty t_1^{\chi} \e^{ -t_1(c_1 + c_2^2 x_2^2)+ \i x_1 t_1 } \d t_1 \nn \\
	&= {\textstyle \frac{1}{2\pi}\frac{1}{(x_1^2 + (c_1 + c_2^2 x_2^2)^2)^{\frac{1}{2}(\chi + \frac{3}{2})} } \exp \{ \i ( \chi+ \frac{3}{2})
	\arctan ( \frac{x_1}{c_1 + c_2^2 x_2^2} ) \}}
\end{align}
and hence  $g$ in \eqref{ex3g} satisfies \eqref{ex3f}.
We note that the representation of $g$ in \cite[(3.7)]{kelb2005} is not explicit; the expression in \eqref{ex3g} was suggested by the derivation of the asymptotics of the fundamental solution of the fractional heat equation on $\Z^2 $ in \cite[proof of Prop.~4.1]{pils2017}.
For $\mbt \in \R^2$, let $g_0(\mbt):=\rho(\mbt)^\chi \ell (\mbt)$ with
\begin{eqnarray} \label{rhoEx3}
	&\rho (\mbt):= |t_1| + |t_2|^2, \qquad q_1 := 1, \ q_2 := 2,  \ Q = \frac{3}{2},\\
	&\ell (\mbt):= \frac{ z^\chi } {2^{\frac 1 2}(2\pi)^{\frac{3}{2}}  c_2  \Gamma(\chi+ \frac{3}{2})}
	\exp \{ - \frac{1}{4c_2^2} ( \frac{1}{z} - 1 ) \} I(t_1 >0),  \qquad \text{where } z := \frac{t_1}{\rho(\mbt)} \in (0,1]. \nn
\end{eqnarray}
Note $\ell$ in \eqref{rhoEx3} is a bounded generalized invariant function for any $\chi > -  \frac{3}{2}$; particularly,
$\ell (\mbt) \to 0$ as $z \downarrow 0$. We have
\begin{equation} \label{ex3g1}
	g(\mbt) =  \rho(\mbt)^\chi \ell(\mbt) (1 + o(1)) = g_0(\mbt)(1+ o(1)),  \qquad |\mbt| \to 0.
\end{equation}
The form of $g_0$ and the asymptotics in \eqref{ex3g1} are similar to \cite[(4.8)]{pils2017} and \cite[(4.11)]{sur2020} in the lattice case.

\begin{proposition}
	The kernel $g$ in \eqref{ex3g} satisfies Assumptions (G)$_\alpha$, (G)$^0_\alpha$ with $g_0$, $q_i$, $i=1,2$, as in \eqref{rhoEx3} for any
	\begin{eqnarray}
	&	-\frac{3}{2\alpha} <  \chi < \frac{3}{2} (1- \frac{1}{\alpha} ),  \qquad \chi \neq 0,
	\end{eqnarray}
	(equivalently,  $P = \frac{3}{3 - 2 \chi}  \in  (\frac{\alpha}{1 + \alpha}, \alpha)$, $P  \neq 1 $) and $\alpha \in (0,2]$.
	As a consequence, the RF $X$ is  \eqref{XLevy3} is well-defined for any
	$M$ satisfying Assumption  (M)$_\alpha$.
\end{proposition}

\begin{proof}
	It suffices to prove the proposition for $\tilde g(\mbt) := t_1^\chi \exp \{-t_1 - \frac{t_2^2}{t_1} \} I(t_1 >0)$ and $\tilde g_0(\mbt) := t_1^\chi \exp \{- \frac{t_2^2}{t_1} \} I(t_1 >0)$, $\mbt \in \R^2_0$.
	Let us verify \eqref{g12bd} for $\mbt = (t_1,t_2) \in \R^2_0$ such that $t_1>0$.
	It is convenient to change the variables as $u:= t_1 >0$, $z := \frac{t_2^2}{t_1} >0$.
	For any $\chi \in \R$, there exists $C>0$ such that $\e^{-z} \le C (1+z)^\chi$ for all $z>0$, hence $\tilde g_0(\mbt) = u^\chi \e^{-z} \le C u^\chi (1+z)^\chi = C \rho(\mbt)^\chi$.
	In a similar way,
	\begin{align*}
	|\partial_1 \tilde g_0(\mbt)|
	&=u^{\chi-1}|\chi   + z|  \e^{-z} \le C (u(1+z))^{\chi -1} = C \rho(\mbt)^{\chi - \frac{1}{q_1}}, \\
	|\partial_2 \tilde g_0(\mbt)|
	&=2u^{\chi- \frac{1}{2}} z^{\frac{1}{2}} \e^{-z}
	\le C  (u(1+z))^{\chi - \frac{1}{2}} =   C\rho(\mbt)^{\chi- \frac{1}{q_2}}, \\
	|\partial_{12} \tilde g_0(\mbt)| &= 2 u^{\chi-\frac{3}{2}} |1-\chi - z| z^{\frac{1}{2}} \e^{-z}
	\le C (u(1+z))^{\chi - \frac{3}{2}} = C \rho(\mbt)^{\chi - Q},
	\end{align*}
	proving \eqref{g12bd}.
	Relations \eqref{gg0} follow from \eqref{g12bd} and $\tilde g(\mbt) = \tilde g_0(\mbt) \e^{-t_1}$ since
	$\e^{-t_1} = 1 + O(t_1)$, $t_1 \downarrow 0$, together with its derivatives.
	Finally, let us verify Assumption (G)$^0_\alpha$.
	After the above-given change of variables,
	$\int_{\R^2} |\tilde g (\mbt)|^\alpha \d \mbt = \int_0^\infty u^{\alpha \chi + \frac{1}{2}} \e^{-\alpha u} \d u \int_0^\infty z^{-\frac{1}{2}} \e^{-\alpha z} \d z < \infty$ since $\chi > -\frac{3}{2\alpha}$.
	Moreover, $\int_{\rho(\mbt)>2} |\partial_1 \tilde g_{0}(\mbt)|^\alpha \e^{-\alpha t_1} I(t_1>0) \d \mbt = \int_0^\infty J (u) u^{\alpha(\chi-1) + \frac{1}{2}} \e^{-\alpha u} \d u < \infty$ since $J(u) := \int_{0 \vee (\frac{2}{u}-1)}^\infty  | \chi+z |^\alpha z^{-\frac{1}{2}} \e^{-\alpha z} \d z \le
	C (u^{\frac{1}{2}-\alpha} \e^{-\frac{\alpha}{u}} I (u \le 1) + I (u > 1))$.
	In a similar way, we can show $\int_{\rho(\mbt)>2} (|\partial_2 \tilde g(\mbt)|^\alpha + |\partial_{12} \tilde g(\mbt)|^\alpha) \d \mbt < \infty$.
\end{proof}

\section{A class of $\alpha$-stable MSS RFs with one of the self-similarity parameters equal 0 or 1}
\label{sec4}

In this section  we define the $\alpha$-stable RFs $\Upsilon_{\alpha,i}$, $\tilde{\Upsilon}_{\alpha,i}$, $i=1,2$, of Table~\ref{tab1} as integrals w.r.t.\ $\alpha$-stable random measure on $\R^2$.
In the Gaussian case $\alpha =2 $ these RFs up to a scale factor coincide with standard FBS defined via the covariance function in \eqref{FBScov} for $\mbH \in (0,1]^2$.
The following definition extends the last covariance to $\mbH \in [0,1]^2 $.

\begin{definition}[\cite{sur2020}] \label{defFBS}
	Standard FBS $B_{\mbH}= \{ B_{\mbH} (\mbt), \, \mbt \in \R^2_+  \}$ with $\mbH =(H_1,H_2)
	\in [0,1]^2$, $H_1 \wedge H_2 = 0$, is defined as a Gaussian process with zero-mean and covariance function $\E B_{\mbH} (\mbt) B_{\mbH} (\mbs) = \prod_{i=1}^2 R_{H_i} (t_i,s_i)$, $\mbt, \mbs \in \R^2_+$, where for $t,s \in \R_+$,
	\begin{align*}
		R_H (t,s) &= {\textstyle \frac{1}{2}} (t^{2H} + s^{2H} - |t-s|^{2H}), \qquad 0<H\le 1,\\
		R_0 (t,s) &= \lim_{H \downarrow 0} R_H (t,s) = 1 - {\textstyle \frac{1}{2}} I(t \neq s).
	\end{align*}
\end{definition}

\begin{remark} \label{rem2}
	The covariance in \eqref{FBScov} implies that the restriction of  FBS  $B_{\mbH}$ to horizontal/vertical line agrees with fractional Brownian motion (FBM) $B_H=\{B_H(t), \, t\in \R_+\}$ with the corresponding Hurst parameter $H = H_i \in (0,1]$, $i=1,2 $.
	Following Definition \ref{defFBS}  we may define FBM $B_H$ with $H=0$ as a Gaussian process on $\R_+$ with zero-mean and the covariance function $\E B_0(t) B_0(s) = 1 - \frac{1}{2} I(t \ne s)$, $t,s \in \R_+$.
	The last process is $H$-SS SI with $H=0$ and satisfies the strange property that $\E (B_0(t)-B_0(s))^2 = 1 = \lim_{H \downarrow 0} |t-s|^{2H}$ for any $t,s \in \R_+$, $t\neq s$.
	It can be represented as $B_0 \eqfdd \{\frac{1}{\sqrt{2}}(W(t) - W(0)), \, t \in \R_+ \}$, where $W(t)$, $t \in [0,\infty)$, is (uncountable) family of \emph{independent} $N(0,1)$ r.v.s.
	See \cite[Examples  1.3.1, 8.2.3]{samo2016}.
	We note that the above $B_0$ is  different from the 'regularized' FBM with $H=0$ defined in \cite[p.2985]{fyod2016}, which is not $0$-self-similar and has a.s.\ continuous paths.
	Non-constant $H$-SS SI processes with $H=0$ are extremely singular (not measurable or `ugly'), see
	\cite[pp.256--257]{samo2016}.
	FBS $B_{\mbH}$ with $H_1\wedge H_2 =0$ and their $\alpha$-stable extensions $\Upsilon_{\alpha,i}$, $i=1,2$, defined below also share these singularity properties and appear to be very unusual objects by most standards in the probability theory.
\end{remark}

We start by defining two classes  $Y_{\alpha,i}$, $\tilde{Y}_{\alpha,i}$, $i=1,2$, of $\alpha$-stable SS SI processes indexed by one-dimensional time parameter.
The corresponding RFs $\Upsilon_{\alpha,i}$, $\tilde{\Upsilon}_{\alpha,i}$, $i=1,2$, indexed by points
of $\R^2_+$  are defined afterwards.
These definitions are completely analogous for $i=1$ and $i=2$ and essentially reduce to exchanging of the coordinate axes.
The processes $Y_{\alpha,i} = \{Y_{\alpha,i}(t), \, t \in \R_+ \}$, $\tilde{Y}_{\alpha,i} = \{\tilde{Y}_{\alpha,i}(t), \, t \in \R_+ \}$, $\alpha \in (0,2]$, are defined by
\begin{equation}\label{TU}
	Y_{\alpha,i}(t)
:= \int_{\R^2} h_{i} (t; \mbu) W_{\alpha}(\d \mbu),   \qquad
	\tilde{Y}_{\alpha,i}(t) := \int_{\R^2} \tilde h_{i} (t; \mbu) W_{\alpha}(\d \mbu), \qquad i=1,2,
\end{equation}
as stochastic integrals in \eqref{chfIW} of deterministic kernel functions
\begin{align} \label{halpha}
	\tilde h_1 (t; \mbu)
	&:=\partial_2  g_0(t \mbe_1 - \mbu) - \partial_2 g_0(-\mbu), \qquad
	\tilde h_2 (t; \mbu) :=\partial_1  g_0(t \mbe_2 - \mbu) - \partial_1 g_0(-\mbu), \nn\\
	h_{i} (t; \mbu) &:= g_0(t\mbe_i -\mbu)-g_0(- \mbu), \quad i=1,2,\qquad h_{0} (\mbt; \mbu)
	:=g_0((-\mbu, \mbt - \mbu]),
\end{align}
where $t\in\R_+$, $\mbt\in \R_+^2$, $\mbu \in \R^2$, $\mbe_1 := (1,0)$, $\mbe_2 := (0,1)$ ($h_0$ is used later to define the limit rectangent RF
arising under well-balanced scaling).
Recall the definition of $P_{c_1,c_2}$ in \eqref{Pcc}.

\begin{proposition} \label{prelim}
	Let  $0 < \alpha \le  2$,  
	$g_0$ be as in \eqref{glim} and
	satisfy the bounds \eqref{g12bdp}, moreover,
	$\frac{\alpha}{1+\alpha} < P  < \alpha$, $P  \ne 1 $.
	Then for any $\mbt \in \R^2_+$, $t \in \R_+ $,
	\begin{itemize}
		\item[(i)] $\|h_0(\mbt; \cdot)\|_\alpha < \infty$;
		\item[(ii)] $\|h_i(t; \cdot)\|_\alpha < \infty$ provided $P_{\frac 1 \alpha, \frac{1+\alpha}{\alpha}} <1$ $(i=1)$ or $P_{\frac{1+\alpha}{\alpha}, \frac{1}{\alpha}} <1$ $(i=2)$ hold;
		\item[(iii)] $\|\tilde h_i(t; \cdot)\|_\alpha < \infty $ provided $P_{\frac{1+\alpha}{\alpha}, \frac{1}{\alpha}} >1$ $(i=1)$ or $P_{\frac{1}{\alpha}, \frac{1+\alpha}{\alpha}} >1$ $(i=2)$ hold.
	\end{itemize}
\end{proposition}

\begin{proof}
	It suffices to consider $\mbt = \1$, $t = 1 $ as well as case $i=1$ in (ii)--(iii) only. We will use the triangle inequality \eqref{rhoineq2} for $\rho_0$ with $p:=\max\{p_1,p_2,1\}$. Let $U_r := \{ \boldsymbol{u} \in \R^2 : \rho_0 (\boldsymbol{u})^{\frac{1}{p}} < r \}$ for $r>2^{\frac{1}{p}}$.

	\smallskip

	\noi (i) We have $\int_{U_r} | h_0 (\1; \mbu) |^\alpha \d \mbu \le C \int_{U_{2 r}} | g_0 (\mbu) |^\alpha \d \mbu$ with $|g_0 (\mbu)| \le C \rho_0 (\mbu)^{P-1}$ by \eqref{g12bdp}, which satisfies $\int_{U_{2r}} \rho_0 (\mbu)^{\alpha(P-1)} \d \mbu < \infty$ by  \eqref{intrho} since $\alpha (P-1) > -P$.
	On the other hand, rewriting $h_0 (\1; \mbu) = \int_{[\boldsymbol{0}, \1]} \partial_{12} g_0 (\mbt - \mbu) \d \mbt$ with $|\partial_{12}g_0(\boldsymbol{t} - \mbu)| \le C \rho_0(\mbu)^{-1}$ for all $\boldsymbol{u} \in U_r^c := \R^2 \setminus U_r$ by \eqref{g12bdp}, \eqref{rhoineq2},
	we have $\int_{U_r^c} |h_0(\1; \mbu)|^\alpha \d \mbu \le C \int_{U_r^c} \rho_0 (\mbu)^{-\alpha} \d \mbu <\infty$ by  \eqref{intrho} since $P < \alpha $.

	\smallskip

	\noi (ii) We have $\int_{U_r} |h_1(1;\mbu)|^\alpha \d \mbu \le C \int_{U_{2r}} |g_0(\mbu)|^\alpha \d \mbu <\infty$ as in the proof of (i).
	Next, rewriting $h_1(1; \mbu) = \int_0^1 \partial_1 g_0( t\mbe_1 - \mbu) \d t$ with $|\partial_1 g_0(t \boldsymbol{e}_1-\boldsymbol{u})| \le C \rho_0(\boldsymbol{u})^{\frac{1}{p_2}-1}$ on $U_r^c$ by \eqref{g12bdp}, \eqref{rhoineq2}, we have $\int_{U_r^c} |h_1 (1;\mbu)|^\alpha \d \mbu \le C \int_{U_r^c} |\rho_0(\mbu)|^{\alpha (\frac{1}{p_2} -  1)} \d \mbu < \infty $  by  \eqref{intrho} since $P_{\frac{1}{\alpha}, \frac{1+\alpha}{\alpha}} < 1$.

	\smallskip

	\noi (iii) We have $\int_{U_r} |\tilde h_1 (1;\mbu)|^\alpha \d \mbu \le C \int_{U_{2r}} |\partial_2 g_0 (\mbu)|^\alpha \d \mbu$ with $|\partial_2 g_0 (\mbu)| \le C \rho_0 (\mbu)^{\frac{1}{p_1}-1}$ by \eqref{g12bdp}, which satisfies $\int_{U_{2r}} \rho_0 (\mbu)^{\alpha (\frac{1}{p_1}-1)} \d \mbu < \infty$ by \eqref{intrho} since $P_{\frac{1+\alpha}{\alpha}, \frac{1}{\alpha}} > 1$.
	Next, rewriting $\tilde h_1(1; \mbu) = \int_0^1 \partial_{12} g_0( t\mbe_1 - \mbu) \d t $  with $|\partial_{12} g_0(t \boldsymbol{e}_1-\boldsymbol{u})| \le C \rho (\boldsymbol{u})^{-1}$ on $U_r^c$ by \eqref{g12bdp}, \eqref{rhoineq2}, we have $\int_{U_r^c} |\tilde h_1(1;\allowbreak \mbu)|^\alpha \d \mbu \le C \int_{U_r^c} \rho_0(\mbu)^{-\alpha} \d \mbu< \infty$ by  \eqref{intrho}  since $P<\alpha$.
	Proposition \ref{prelim} is proved.
\end{proof}

Let
\begin{eqnarray}\label{defH}
	&H_{\alpha,1}:=\frac{1+ \alpha}{\alpha} (1  +  \frac{p_1}{p_2} ) - p_1, \qquad
	\tilde H_{\alpha,1} := \frac{1+\alpha}{\alpha} + \frac{p_1}{\alpha p_2} - p_1, \\
	&H_{\alpha,2}:=\frac{1+ \alpha}{\alpha} (1  +  \frac{p_2}{p_1} ) - p_2, \qquad
	\tilde H_{\alpha,2} := \frac{1+\alpha}{\alpha} + \frac{p_2}{\alpha p_1} - p_2. \nn
\end{eqnarray}
Note the equivalencies:
\begin{eqnarray*}
	P> \mbox{$\frac{\alpha}{1+\alpha}$} \ \ \& \ \
	P_{\frac{1}{\alpha},\frac{1+\alpha}{\alpha}} <1  &\Longleftrightarrow&   0< H_{\alpha,1} <1, \\
	P < \alpha \ \ \& \ \  P_{\frac{1}{\alpha}, \frac{1+\alpha}{\alpha}} >1 &\Longleftrightarrow&  0< \tilde H_{\alpha,2} < 1, \\
	P_{\frac{1}{\alpha}, \frac{1+\alpha}{\alpha}} = 1  &\Longleftrightarrow&  H_{\alpha,1} =1  \ \ \& \ \  \tilde H_{\alpha,2} =0.
\end{eqnarray*}
Similar equivalencies hold for $\tilde H_{\alpha,1}$ and $H_{\alpha,2}$ by symmetry.
Also note that $\tilde H_{\alpha,1} = \tilde H_{\alpha,2}=1 $ when $P=\alpha $ and  $H_{\alpha,1} = H_{\alpha,2}=0 $ when $P= \frac{\alpha}{1+\alpha}. $

\begin{corollary} \label{propTU}
	Let $g_0$, $p_i$, $i=1,2$, satisfy the conditions of Proposition \ref{prelim}.
	Then the random processes in \eqref{TU} are well-defined, $H$-SS SI and have $\alpha$-stable finite dimensional distributions in the parameter regions indicated in Table \ref{tab2} below:
	\begin{table}[htbp]
		\begin{center}
			\begin{tabular}{crrrr}
				&$Y_{\alpha, 1} $ & ${\tilde Y}_{\alpha, 2}  $ & $Y_{\alpha, 2} $  & ${\tilde Y}_{\alpha, 1} $ \\
				\hline
				Parameter region & $P_{\frac{1}{\alpha}, \frac{1+\alpha}{\alpha}} < 1$ & $P_{\frac{1}{\alpha}, \frac{1+\alpha}{\alpha}} > 1$   &  $P_{\frac{1+\alpha}{\alpha}, \frac{1}{\alpha}} < 1 $
				&   $P_{\frac{1+\alpha}{\alpha}, \frac{1}{\alpha}} > 1$ \\
				$H$    & $H_{\alpha,1}$     & $\tilde H_{\alpha,2}$   &  $H_{\alpha,2}$   & $\tilde H_{\alpha,1}$\\
				\hline
			\end{tabular}
			\caption{Parameter regions and self-similarity indices of $\alpha$-stable random processes $Y_{\alpha, i}, {\tilde Y}_{\alpha, i}$, $i=1,2$,
				in \eqref{TU}.}\label{tab2}
		\end{center}
	\end{table}
	
	\noi Particularly, in the Gaussian case $\alpha =2 $ the processes  $Y_{2, i}$, ${\tilde Y}_{2,i}$, $i=1,2$, agree with corresponding FBM, viz.,
	\begin{equation}\label{TUfbm}
		Y_{2,i} \eqfdd \sigma_{i} B_{H_{2, i}},  \qquad {\tilde Y}_{2,i} \eqfdd
		\tilde \sigma_{i} B_{\tilde H_{2,i}},  \qquad  i=1,2,
	\end{equation}
	where $ \sigma^2_{i} := \| h_i (1;\cdot) \|_2^2$, $\tilde \sigma^2_{i} :=\| \tilde h_i (1; \cdot) \|_2^2$, $i=1,2$.
\end{corollary}

\begin{proof}
	The fact that the processes in \eqref{TU} are well-defined and have $\alpha$-stable
	distribution follow from Proposition \ref{prelim} and general properties of stochastic integrals w.r.t.\ $\alpha$-stable random measure \cite{samo1994}.
	The stationarity of increments property is a consequence of the form of the integrands in \eqref{halpha} and the invariance of $W_\alpha $ w.r.t.\ to shifts in $\R^2 $.
	Similarly, the  $H$-SS property follows from change of variables in the stochastic integral; particularly, $\{ Y_{\alpha,1} (\lambda t), \, t \in \R_+ \} \eqfdd \{\lambda^{H_{\alpha,1}} Y_{\alpha,1}(t), \, t \in \R_+\}$ from
	\begin{align*}
	Y_{\alpha,1}(\lambda t)&=\int_{\R^2} (g_0(\lambda t - u_1, - u_2) - g_0(-u_1,-u_2))  W_{\alpha}(\d \mbu) \\
	&=\int_{\R^2} (g_0(\lambda t - \lambda u_1, - \lambda^{\frac{p_1}{p_2}}  u_2) - g_0(-\lambda u_1,-\lambda^{\frac{p_1}{p_2}}  u_2))  W_{\alpha}(\d \lambda u_1, \d \lambda^{\frac{p_1}{p_2}} u_2)
	\end{align*}
	by using $g_0(\lambda t_1, \lambda^{\frac{q_1}{q_2}} t_2) = \lambda^{\chi q_1} g_0(\mbt)$, $\chi q_1 = p_1(P-1)$, $ \frac{q_1}{q_2 }=\frac{p_1}{p_2}$, see \eqref{fkappa}, \eqref{pp}, and $W_\alpha ( \d \lambda u_1, \allowbreak \d \lambda^{\frac{p_1}{p_2}} u_2) \allowbreak \eqfdd \allowbreak \lambda^{\frac{1}{\alpha}(1 + \frac{p_1}{p_2})} W_\alpha (\d \mbu)$.
	The self-similarity $\{ {\tilde Y}_{\alpha,1} (\lambda t), \, t \in \R_+ \} = \{  \lambda^{\tilde H_{\alpha,1}} {\tilde Y}_{\alpha,1} (t), \, t \in \R_+ \}$ follows analogously from
	\begin{equation*}
	{\tilde Y}_{\alpha,1}(\lambda t)
	=\int_{\R^2} (\partial_2 g_0(\lambda t - \lambda u_1, - \lambda^{\frac{p_1}{p_2}} u_2) - \partial_2 g_0(-\lambda u_1,-\lambda^{\frac{p_1}{p_2}} u_2))  W_{\alpha}(\d \lambda u_1, \d \lambda^{\frac{p_1}{p_2}} u_2).
	\end{equation*}
	Finally, \eqref{TUfbm} is a consequence of the well-known characterization of Gaussian $H$-SS SI processes \cite{samo1994}.
\end{proof}

Next we define two classes of RFs on $\R^2_+$ following the definitions in \eqref{TU}.
The first class is defined as 
\begin{equation}\label{frakUdef}
	{\tilde \Upsilon}_{\alpha,1}(\mbt) := t_2
	{\tilde Y}_{\alpha,1}(t_1), \qquad  {\tilde \Upsilon}_{\alpha,2}(\mbt) := t_1 {\tilde Y}_{\alpha,2}(t_2), \qquad
\mbt \in \R^2_+.
\end{equation}
Thus, \eqref{frakUdef} are nothing but simple line extensions of the processes ${\tilde Y}_{\alpha, i}$, $i=1,2$: 
for any fixed $t_1$, ${\tilde \Upsilon}_{\alpha,1}(\mbt)$ is a half-line in $t_2$ with
random slope ${\tilde Y}_{\alpha, 1}(t_1)$, the definition of ${\tilde \Upsilon}_{\alpha,2}(\mbt)$ being analogous.

The definition of the second class of  RFs (corresponding to  $Y_{\alpha, i}$, $i=1,2 $) is more involved.
For $m \in \N$, let $Y^{(j)}_{\alpha,1}$, $j=0,1,\dots, m$, be independent copies of $Y_{\alpha, 1}$ in \eqref{TU}.
Pick $(t_{j,1},t_j) \in \R^2_+$, $j=1, \dots, m$, with different ordinates $t_1 < \cdots < t_m$.
Define finite dimensional distribution of a RF $\Upsilon_{\alpha,1}$:
\begin{equation}\label{T1def}
	\big\{ \Upsilon_{\alpha,1}(t_{j,1},t_j), \, j =1, \dots, m \big\}
	\eqd \big\{Y^{(j)}_{\alpha,1}(t_{j,1}) - Y^{(0)}_{\alpha,1}(t_{j,1}), \, j=1, \dots, m \big\}.
\end{equation}
This definition extends to arbitrary finite collection of points in $\R^2_+$. Namely, to extend \eqref{T1def}
let $(t_{j,k}, t_j) \in \R^2_+$, $k=1,\dots, n_j$, $j=1, \dots, m$, be given with $t_1 < \cdots < t_m$.
Then
\begin{equation}\label{frakT1def}
	\big\{ \Upsilon_{\alpha,1}(t_{j,k}, t_j), \, k=1,\dots,n_j, \, j =1, \dots, m \big\} \eqd \big\{ Y^{(j)}_{\alpha,1}(t_{j,k}) - Y^{(0)}_{\alpha,1}(t_{j,k}), \, k=1,\dots,n_j, \, j =1, \dots,m \big\}. \nn
\end{equation}
Similarly, we define a RF $\{ \Upsilon_{\alpha,2}(\mbt), \, \mbt \in \R^2_+ \}$ such that for each finite collection of points $(t_j,t_{j,k}) \in \R^2_+$, $k=1,\dots,n _j$, $j=1,\dots,m$, with abscissas $t_1<\dots < t_m$:
\begin{equation} \label{T2def}
	\big\{ \Upsilon_{\alpha,2}(t_j,t_{j,k}),\, k=1,\dots,n_j, \, j =1, \dots, m \big\}
	\eqd \big\{ Y^{(j)}_{\alpha,2}(t_{j,k}) - Y^{(0)}_{\alpha,2}(t_{j,k}),\, k=1,\dots,n_j, \, j=1, \dots, m \big\},
\end{equation}
where $Y^{(j)}_{\alpha,2}$, $j=0,1,\dots, m$, are independent copies of $Y_{\alpha, 2}$ in \eqref{TU}.
Finally, for $\mbt \in \R^2_+$, set
\begin{equation}\label{V0}
\Upsilon_{\alpha,0}(\mbt) :=  \int_{\R^2} h_0(\mbt; \mbu) W_\alpha(\d \mbu).
\end{equation}

The following corollary summarises the properties of the introduced RFs and details their MSS indices mentioned in the beginning of this section.

\begin{corollary} \label{corsum}
	Let $g_0$, $\alpha$, $p_i$, $i=1,2$, satisfy the conditions in Proposition \ref{prelim}.
	Then:
	\begin{itemize}
		\item[(i)] RF $\Upsilon_{\alpha,0}$ in \eqref{V0} is well-defined;
		\item[(ii)] RFs $\Upsilon_{\alpha, i}$, ${\tilde \Upsilon}_{\alpha, i}$ are well-defined in the parameter regions shown in Table \ref{tab2} for $Y_{\alpha, i}, {\tilde Y}_{\alpha, i}$ respectively, $i=1,2$;
		\item[(iii)] RFs $\Upsilon_{\alpha,0}$, $\Upsilon_{\alpha, i}$, ${\tilde \Upsilon}_{\alpha, i}$, $i=1,2$, have $\alpha$-stable finite dimensional distributions and stationary rectangular increments;
		\item[(iv)] ${\tilde \Upsilon}_{\alpha, i}$ 
		is MSS RF with index $\mbH = (\tilde H_{\alpha,1},1)$  if $i=1$, and $\mbH = (1,\tilde H_{\alpha,2})$ if $i=2$;
		\item[(v)] $\Upsilon_{\alpha, i}$ is MSS RF with index $\mbH = (H_{\alpha,1}, 0)$ if $i=1$, and $\mbH = (0,H_{\alpha,2})$ if $i=2$;
		\item[(vi)] In the Gaussian case $\alpha =2 $ RFs  $\Upsilon_{2, i}$, ${\tilde \Upsilon}_{2,i}$, $i=1,2$, agree with FBS, viz.,
		\begin{equation}
			\Upsilon_{2,1} \eqfdd \sigma_{1} B_{(H_{2,1},0)},  \quad \Upsilon_{2,2} \eqfdd \sigma_{2} B_{(0,H_{2,2})}, \quad
			{\tilde \Upsilon}_{2,1} \eqfdd
			\tilde \sigma_{1} B_{(\tilde H_{2,1},1)}, \quad
			{\tilde \Upsilon}_{2,2} \eqfdd
			\tilde \sigma_{2} B_{(1,\tilde H_{2,2})},
		\end{equation}
		where $ \sigma^2_{i}$, $\tilde \sigma^2_{i}$, $i=1,2$, are given in \eqref{TUfbm}.
	\end{itemize}
\end{corollary}

\begin{proof}
	Most facts in Corollary \ref{corsum} follow from Proposition \ref{prelim} and the definitions of the introduced RFs.
	Let us check the MSS property of $\Upsilon_{\alpha, i}$, viz.,
	\begin{equation}\label{Tmss}
	\big\{ \Upsilon_{\alpha, i}(\lambda_1 t_1, \lambda_2 t_2), \, \mbt \in \R^2_+ \big\}\eqfdd
	\big\{ \lambda_i^{H_{\alpha,i}} \Upsilon_{\alpha, i} (\boldsymbol{t}), \, \mbt \in \R^2_+ \big\}, \qquad \forall \lambda_1 \in \R_+, \ \lambda_2 \in \R_+,
	\end{equation}
	for $i=1,2$.
	Let  $\{ (t_{j,k}, t_j ) \in \R^2_+, \, k=1,\dots, n_j, \, j=1,\dots, m \}$ be an arbitrary collection of points with $t_1<t_2<\dots<t_{m}$.
	Then the rescaled collection $\{ (\lambda_1 t_{j,k}, \lambda_2 t_j ) \in \R^2_+, \, k=1,\dots, n_j, \, j=1,\dots, m \}$ satisfies the same property: $\lambda_2 t_1< \lambda_2 t_2<\dots< \lambda_2 t_{m}$.
	Therefore, with $Y_{\alpha,1}^{(j)}$, $j=0,1,\dots,n$, independent copies of $H_{\alpha,1}$-SS process $Y_{\alpha, 1}$ in \eqref{TU},
	\begin{align*}\label{frakTprime}
	&\big\{ \Upsilon_{\alpha,1}(\lambda_1 t_{j,k}, \lambda_2 t_j), \, k=1,\dots, n_j, \, j =1, \dots, m\big\}\\
	&\qquad\eqd \big\{ Y^{(j)}_{\alpha,1}( \lambda_1 t_{j,k}) - Y^{(0)}_{\alpha,1}(\lambda_1 t_{j,k}), \, k=1, \dots, n_j, \, j=1,\dots, m
	\big\} \\
	&\qquad\eqd \big\{ \lambda_1^{H_{\alpha,1}} ( Y^{(j)}_{\alpha,1}(t_{j,k}) - Y^{(0)}_{\alpha,1}(t_{j,k}) ), \, k=1, \dots, n_j, \,
	j=1,\dots, m
	\big\} \\
	&\qquad\eqd \big\{\lambda_1^{H_{\alpha,1}} \Upsilon_{\alpha,1}(t_{j,k},t_j), \, k=1, \dots, n_j, \, j=1,\dots, m \big\},
	\end{align*}
	proving \eqref{Tmss} for $i=1$.
\end{proof}

\section{Rectangent limits of L\'evy driven fractional RF}
\label{sec5}

The following Theorem \ref{mainthm} is the main result of our paper.

\begin{theorem} \label{mainthm}
	Let L\'evy driven fractional RF $X$ in \eqref{def:X}  satisfy Assumptions (G)$^0_\alpha$, (G)$_\alpha$ and (M)$_\alpha$; $0< \alpha \le 2$, $\frac{\alpha}{1+\alpha} < P < \alpha$, $P\ne 1$, $P_{\frac{1}{\alpha},\frac{1+\alpha}{\alpha}} \ne 1$, $P_{\frac{1+\alpha}{\alpha},\frac{1}{\alpha}} \ne  1$.
	Then the $\gamma$-rectangent RF in \eqref{Vlim} exists for any $\gamma >0$, $\mbt_0 \in \R^2_+$ and satisfies the trichotomy
	\begin{equation} \label{V3}
		V_\gamma =
		\begin{cases} V_+, &\gamma > \gamma_0, \\
			V_-, &\gamma < \gamma_0, \\
			V_0, &\gamma = \gamma_0,
		\end{cases}
	\end{equation}
	with $\gamma_0 = \frac{q_1}{q_2} = \frac{p_1}{p_2}$, 
	$V_0 := \Upsilon_{\alpha,0}$ defined in  \eqref{V0} and
	\begin{align}
		V_-
		:=
		\begin{cases}
			{\tilde \Upsilon}_{\alpha,2}, &P_{\frac{1}{\alpha},\frac{1+\alpha}{\alpha}} > 1, \\
			\Upsilon_{\alpha,1}, &P_{\frac{1}{\alpha},\frac{1+\alpha}{\alpha}} < 1,
		\end{cases} \qquad
		V_+
		:=
		\begin{cases}
			{\tilde \Upsilon}_{\alpha,1}, &P_{\frac{1+\alpha}{\alpha}, \frac{1}{\alpha}} > 1, \\
			\Upsilon_{\alpha,2}, &P_{\frac{1+\alpha}{\alpha},\frac{1}{\alpha}} < 1.
		\end{cases}
	\end{align}
	The normalization $d_{\lambda,\gamma} = \lambda^{H(\gamma)} $ in \eqref{Vlim}  is defined in the proof of Theorem \ref{mainthm}.
\end{theorem}

\begin{proof}
		\underline {Case $\gamma = \gamma_0$.}
	Let $H(\gamma_0) =  (1+ \gamma_0) \frac{1 + \alpha}{\alpha} - p_1 = (1+ \gamma_0) \frac{1}{\alpha} + q_1 \chi$, $\Gamma_0 := \operatorname{diag}(1, \gamma_0) $. Then $\lambda^{-H(\gamma_0)} X((\boldsymbol{0}, \lambda^{\Gamma_0} \boldsymbol{t}])
	= \int_{\R^2} f_{\lambda} (\mbt; \mbu) M(\d \mbu)$, where
	$f_{\lambda} (\mbt; \mbu) = \lambda^{-H(\gamma_0)} g ((-\mbu, \lambda^{\Gamma_0}\mbt - \boldsymbol{u}])$, $\mbt \in \R^2_+$, $\mbu \in  \R^2$.
	It suffices to show that for any $\theta_i \in \R$, $\mbt_i \in \R^2_+$, $i=1, \dots, m$, $m \ge 1$,
	\begin{equation}
	\sum_{i=1}^m \theta_i \int_{\R^2} f_{\lambda} (\mbt_i; \mbu) M(\d \mbu) \ \limd \  \sum_{i=1}^m \theta_i\int_{\R^2}  h_0(\mbt_i; \mbu) W_\alpha(\d \mbu),
	\end{equation}
	where $h_0(\mbt; \mbu) = g_0((-\mbu, \mbt - \mbu])$, $\boldsymbol{t} \in \R^2_+$, $\boldsymbol{u} \in \R^2$, as in \eqref{halpha}. Using Proposition \ref{offp} this follows from
	\begin{equation} \label{normf}
	\Big\| \sum_{i=1}^m \theta_i (f^\dagger_\lambda (\mbt_i; \cdot) - h_0(\mbt_i; \cdot))\Big\|_{\alpha} \to 0.
	\end{equation}
	Obviously, it suffices to prove \eqref{normf} for $m= \theta_1 = 1$, $\mbt_1 = \mbt $. Letting $\mu_1 = 1, \mu_2 = \gamma_0$, we have $f^\dagger_\lambda (\mbt; \mbu) = \lambda^{-\chi q_1} g_0((- \lambda^{\Gamma_0} \mbu, \lambda^{\Gamma_0}(\mbt - \boldsymbol{u})]) + o(1) = h_0(\mbt; \mbu) + o(1)$ for any $\mbu \neq \boldsymbol{0}, (0,t_2), (t_1,0), \mbt$, by \eqref{gg0}. Consequently, we have \eqref{normf} by the DCT using \eqref{gg0}, \eqref{g12bd} similarly as in the proof of Proposition~\ref{prelim}~(i).
	Indeed, the above relations imply the existence of $\epsilon >0$ such that for all $0< |\mbu| < \epsilon$,
	\begin{equation}\label{grbd}
	|g(\mbu)| \le  C\rho(\mbu)^\chi, \qquad |\partial_i g(\mbu)| \le C\rho(\mbu)^{\chi - \frac{1}{q_i}},  \quad i=1,2,  \qquad
	|\partial_{12} g(\mbu)| \le C \rho(\mbu)^{\chi - Q}.
	\end{equation}
	Letting $\mbt = \1$ and using \eqref{grbd} we have that for all
	$\lambda>0$ small enough,
	$|\boldsymbol{u}| < \epsilon$,
	\begin{align*}
	|f^\dagger_\lambda (\1; \mbu)| &\le C \lambda^{-\chi q_1} \big(\rho(\lambda^{\Gamma_0}(\1 - \mbu))^\chi + \rho(\lambda^{\Gamma_0}(\mbe_2-\mbu))^\chi + \rho(\lambda^{\Gamma_0}(\mbe_1-\mbu) )^\chi +
	\rho(-\lambda^{\Gamma_0} \mbu)^\chi \big)\\
	&= C \big(\rho(\1 - \mbu)^\chi + \rho(\mbe_2-\mbu)^\chi + \rho(\mbe_1-\mbu)^\chi + \rho(-\mbu)^\chi \big),
	\end{align*}
	where the dominating function is $\alpha$-integrable  on $|\mbu| < \epsilon$, see \eqref{intrho}. Similarly, for all $|\lambda^{\Gamma_0} \mbu| < \epsilon \le |\mbu|$,
	\begin{equation*}
	|f^\dagger_\lambda (\1; \mbu) |
	=  \Big| \lambda^{-\chi q_1} \int_{[\0,\1]} \partial_{12} g(\lambda^{\Gamma_0}(\mbs - \mbu)) \d \mbs \Big|
	\le  C \int_{[\0,\1]} \rho(\mbs-\mbu)^{\chi-Q} \d \mbs,
	\end{equation*}
	where the dominating function is $\alpha$-integrable on $|\mbu| \ge \epsilon$.
	Finally, we have $\int_{|\lambda^{\Gamma_0} \mbu| \ge \epsilon} | f^\dagger_\lambda (\1; \mbu) |^\alpha \d \mbu = \lambda^{\alpha(1+\gamma_0-H(\gamma_0))} I_\lambda = o(1)$, because $1+\gamma_0-H(\gamma_0) = q_1(Q(1-\frac{1}{\alpha})-\chi) = \frac{p_1}{\alpha} (\alpha- P) > 0$ and
	\begin{align*}
	I_\lambda &= \int_{|\mbu| \ge \epsilon} \Big| \int_{[\0,\1]} \partial_{12} g (\lambda^{\Gamma_0} \mbs - \mbu) \d \mbs \Big|^\alpha \d \mbu\\
	&\le  \int_{|\mbu| \ge  \epsilon} \int_{[\0,\1]} |\partial_{12} g(\lambda^{\Gamma_0} \mbs - \mbu)|^\alpha \d \mbs \d \mbu
	\le \int_{|\mbu| \ge \frac{\epsilon}{2}} |\partial_{12} g(\mbu)|^\alpha \d \mbu < \infty
	\end{align*}
for $1\le \alpha \le 2 $ follows by Jensen's inequality from Assumption (G)$^0_\alpha$; for $0< \alpha < 1 $
relation $I_\lambda  = O(1)$ is a  consequence of the monotonicity of the dominating function $\bar g_{12}$ in  \eqref{G02},
Assumption (G)$^0_\alpha$.
This completes the proof of \eqref{normf}.

	\medskip

	\noi \underline {Case $\gamma < \gamma_0, P_{\frac{1}{\alpha},\frac{1+\alpha}{\alpha}} < 1$.} Let $H(\gamma) =  H_{\alpha,1}$.
	It suffices to prove the convergence
	\begin{equation}\label{convTlim}
	\sum_{i=1}^n \theta_i \lambda^{-H_{\alpha,1}}  X((\boldsymbol{0}, \lambda^{\Gamma} \boldsymbol{t}_i])
	\limd \sum_{i=1}^n \theta_i \Upsilon_{\alpha,1}(\mbt_i)
	\end{equation}
	for any $\theta_i \in \R$, $\mbt_i \in \R^2_+$, $i=1,\dots, n$, $n \ge 1$.
	Write the collection $\{\mbt_i, \, i=1,\dots, n \}= \{ (t^{(j)}_{k,1}, t^{(j)}_2), \, k=1,\dots, n^{(j)}, \, j=1,\dots m \}$ as a union of $m$ sub-collections of points in $\R^2_+$ with different ordinates, viz., $t^{(1)}_2< t^{(2)}_2 < \dots < t^{(m)}_2$ and $\sum_{j=1}^m n^{(j)} = n$.
	Set $t_2^{(0)} := 0$ and first let $m=1$.
	For $\mbt = (t_1, t_2) \in \R^2_+$ such that $t_2 = t_2^{(1)}$, split $\lambda^{-H_{\alpha,1}} X((\0, \lambda^\Gamma \mbt]) = {\bar X}^{(1)}_\lambda (t_1) - {\bar X}^{(0)}_\lambda (t_1)$, where
	\begin{equation}\label{Sdecomp}
	{\bar X}^{(j)}_\lambda (t_1) = \int_{\R^2} f_\lambda (t_1;\mbu- \lambda^\gamma t_2^{(j)} \mbe_2) M(\d \mbu), \qquad j=0,1,
	\end{equation}
	with $f_\lambda (t_1;\mbu) = \lambda^{-H_{\alpha,1}} (g (\lambda t_1 \mbe_1-\mbu) - g (-\mbu))$, $\mbu \in \R^2$.
	Note the two processes $\{{\bar X}^{(j)}_\lambda (t_1), \, t_1 \in \R_+\}$, $j=0,1$, 
	are identically distributed and asymptotically independent as it follows from the argument below.
	Let us prove that as $\lambda \downarrow 0$,
	\begin{equation}\label{Y0}
	\big({\bar X}^{(0)}_\lambda (t_1), {\bar X}^{(1)}_\lambda (t_1)\big) \limfdd \big(Y^{(0)}_{\alpha,1} (t_1), Y^{(1)}_{\alpha,1} (t_1)\big), 
	\end{equation}
	where $Y^{(j)}_{\alpha,1}$, $j=0,1$, are independent copies of $Y_{\alpha,1}$ in \eqref{TU}.
	For this purpose, write $f_\lambda (t_1;\mbu)$ as a sum of
	\begin{align}\label{def:flambdai}
	f_{\lambda,0}(t_1;\mbu) &= f_\lambda (t_1;\mbu)  I(|u_2|< \epsilon \lambda^\gamma, |\mbu|< \epsilon),\\
	f_{\lambda,1} (t_1;\mbu) &= f_\lambda (t_1;\mbu) I(|u_2| \ge \epsilon\lambda^\gamma, |\mbu| < \epsilon),\nn\\
	f_{\lambda,2} (t_1;\mbu) &= f_\lambda (t_1; \mbu) I(|\mbu| \ge \epsilon)\nn
	\end{align}
	for all $ \mbu \in \R^2$ and some small enough $\epsilon>0$.
	Then, using the fact that the processes $\{\int_{\R^2} f_{\lambda,0} (t_1; \mbu-\lambda^\gamma t_2^{(j)} \mbe_2) M(\d \mbu), \, t_1 \in \R_+ \} $, 
	$j=0,1$, are independent and have the same distribution, relation \eqref{Y0} follows if we show that 
	with $h_1$ defined by \eqref{halpha},
	\begin{equation}\label{Sone}
	\int_{\R^2} f_{\lambda,0} (t_1;\mbu) M(\d \mbu) \limfdd \int_{\R^2} h_1 (t_1;\mbu) W_\alpha(\d \mbu)= Y_{\alpha,1} (t_1) \ \ \text{and} \ \ \| f_{\lambda,i} (t_1; \cdot) \|_\alpha = o(1), \ \ i=1,2.
	\end{equation}
	The first relation in \eqref{Sone} follows by Proposition \ref{offp} with $\mu_1=1$, $\mu_2 = \gamma_0$ from $\|f^\dagger_{\lambda,0} (t_1; \cdot) - h_1 (t_1; \cdot) \|_\alpha = o(1)$, which in turn follows by the DCT.
	Indeed, since $\gamma<\gamma_0$ we have $f^\dagger_{\lambda,0} (t_1;\mbu)= f^\dagger_\lambda (t_1;\mbu) I(| u_2|<\epsilon \lambda^{\gamma-\gamma_0}, |\lambda^{\Gamma_0} \mbu| < \epsilon) \sim f^\dagger_\lambda (t_1;\mbu)$, where
	\begin{align*}
	f^\dagger_\lambda (t_1;\mbu) = \lambda^{-q_1 \chi} (g(\lambda^{\Gamma_0} (t_1 \mbe_1 -\mbu)) - g(-\lambda^{\Gamma_0} \mbu))
	\sim g_0(t_1 \mbe_1 - \mbu)-g_0(-\mbu) = h_1 (t_1;\mbu)
	\end{align*}
	using \eqref{gg0} and $\lambda^{-q_1 \chi} o(\rho( \lambda^{\Gamma_0} \mbu)^\chi) = o(1)$ for any $\mbu \neq \0$, $t_1 \mbe_1$.
	Moreover,
	\begin{align*}
	|f^\dagger_{\lambda,0} (t_1; \mbu)| \le|f^\dagger_\lambda (t_1;\mbu)| I (|\lambda^{\Gamma_0} \mbu| < \epsilon)
	&\le C \Big( \big( \rho_0 (t_1 \mbe_1-\mbu)^{P-1} + \rho_0 (-\mbu)^{P-1} \big) I\big( |\mbu| < 1 \big)\\
	&\qquad+ \int_0^{t_1} \rho_0 (s\mbe_1 - \mbu)^{\frac{1}{p_2}-1} \d s  I\big( |\mbu| \ge 1 \big) \Big),
	\end{align*}
	because $|\lambda^{\Gamma_0} \mbu| < \epsilon$ implies $|\lambda^{\Gamma_0} (t_1 \mbe_1 -\mbu)| < 2\epsilon$, and consequently we can use \eqref{gg0}, \eqref{g12bdp} and \eqref{g12bd} to bound the terms of $f^\dagger_\lambda (t_1; \mbu)$ and integrand of rewritten $f^\dagger_\lambda (t_1;\mbu) = \lambda^{1-\chi q_1} \int_0^{t_1} \partial g_1 (\lambda^{\Gamma_0} (s \mbe_1 - \mbu) ) \d s$. The above function, which dominates $f^\dagger_{\lambda,0} (t_1; \cdot)$, belongs to $L_\alpha (\R^2)$, see the proof of Proposition \ref{prelim}~(ii).
	The above domination also holds for $f_{\lambda,1}^\dagger (t_1; \cdot)$, except that now we have
	\begin{equation*}
	f^\dagger_{\lambda,1} (t_1;\mbu) =
	f^\dagger_\lambda (t_1;\mbu) I(| u_2| \ge \epsilon \lambda^{\gamma-\gamma_0}, |\lambda^{\Gamma_0} \mbu| < \epsilon) = o(1)
	\end{equation*}
	point-wise (since $\lambda^{\gamma-\gamma_0} \to \infty$), resulting in $\| f_{\lambda,1}^\dagger (t_1; \cdot) \|_\alpha = \| f_{\lambda,1} (t_1; \cdot) \|_\alpha = o(1)$.
	Finally, rewriting $
	\int_{\R^2} | f_{\lambda,2} (t_1; \mbu) |^\alpha \d \mbu 
	=\int_{|\mbu| \ge \epsilon} | f_\lambda (t_1; \mbu) |^\alpha \d \mbu$ with $f_\lambda (t_1;\mbu) =\lambda^{1-H_{\alpha,1}} \int_0^{t_1} \partial_1 g(\lambda s \mbe_1 - \mbu) \d s$ we get $\|  f_{\lambda,2} (t_1; \cdot ) \|_\alpha
	= o(1)$ since $1-H_{\alpha,1} =p_1 (1-P_{\frac{1}{\alpha}, \frac{1+\alpha}{\alpha}})> 0$ and $\int_{|\mbu| \ge \epsilon} |\int_0^{t_1} \partial_1 g(\lambda s \mbe_1 - \mbu) \d s |^\alpha \d \mbu < C $ follows from Assumption (G)$^0_\alpha$ similarly as in the case $\gamma = \gamma_0$ above.
	This completes the proof of \eqref{Y0} or \eqref{convTlim} for $m=1$.

	The above argument easily extends to the general case $1 \le m \le n$ in  \eqref{convTlim}.
	Let $1\le j \le m$. For $\mbt = (t_1,t_2) \in \R^2_+$ such that $t_2 = t_2^{(j)}$, consider a similar decomposition as in \eqref{Sdecomp}, viz., $\lambda^{-H_{\alpha,1}} X((\boldsymbol{0}, \lambda^{\Gamma}\mbt ]) = \bar X_\lambda^{(j)} (t_1) - \bar X_\lambda^{(0)} (t_1)$, where
	\begin{equation*}
	\bar X^{(i)}_\lambda(t_1)
	=  \int_{\R^2} f_\lambda (t_1; \mbu - \lambda^\gamma t_2^{(i)} \mbe_2) M(\d \mbu), \qquad i= 0,j,
	\end{equation*}
	with $f_\lambda (t_1; \mbu) = \sum_{i=0}^2 f_{\lambda,i}(t_1; \mbu)$ for all $\mbu \in \R^2$ and some $\epsilon>0$ as in \eqref{def:flambdai}.
	If $\epsilon>0$ is small enough, then $\{ \int_{\R^2} f_{\lambda,0} (t_1;\mbu- \lambda^\gamma t_2^{(j)} \mbe_2) M(\d \mbu), \, t_1 \in \R_+ \}$, $j=0, \dots, m$, are independent, moreover,
	\begin{equation*}
	\int_{\R^2} f_{\lambda,0} (t_1;\mbu- \lambda^\gamma t_2^{(j)} \mbe_2) M(\d \mbu) \limfdd Y_{\alpha,1}^{(j)} (t_1), \qquad j=0,\dots, m,
	\end{equation*}
	whereas other terms are negligible as in \eqref{Sone}, which proves \eqref{convTlim}.

	\medskip

	\noi \underline {Case $\gamma < \gamma_0, P_{\frac{1}{\alpha},\frac{1+\alpha}{\alpha}} > 1$.} Let $H(\gamma) = 1 + \gamma \tilde H_{\alpha,2}$. It suffices to prove one-dimensional convergence
	\begin{equation}\label{convU}
	\lambda^{-1- \gamma \tilde H_{\alpha,2}} X((\boldsymbol{0}, \lambda^{\Gamma} \boldsymbol{t}]) \limd t_1 {\tilde Y}_{\alpha,2}(t_2)
	=  t_1 \int_{\R^2} \tilde h_{2}(t_2;\mbu)  W_\alpha(\d \mbu)
	\end{equation}
	for arbitrary $\mbt = (t_1,t_2) \in \R^2_+$ with $\tilde h_{2}(t; \mbu) = \partial_1  g_0(t \mbe_2 - \mbu) - \partial_1 g_0(-\mbu)$, $t\in \R_+$, $\mbu \in \R^2$ in \eqref{halpha}.
	Since the l.h.s.\ of \eqref{convU} writes as $\int_{\R^2} f_{\lambda} (\mbt; \mbu) M(\d \mbu)$ then using Proposition \ref{offp} with $\mu_1 := \frac{\gamma}{\gamma_0}$, $\mu_2 := \gamma $ it suffices to prove
	\begin{equation}\label{f3}
	\| f^\dagger_{\lambda} (\mbt; \cdot) - t_1 \tilde h_{2}(t_2;\cdot)\|_\alpha = o(1),
	\end{equation}
	where with
	$\Gamma'_0 = \operatorname{diag}(\frac{\gamma}{\gamma_0}, \gamma)$, $\lambda' = \lambda^{1-\frac{\gamma}{\gamma_0}} \downarrow 0$,
	\begin{align*}
	f^\dagger_{\lambda} (\mbt; \mbu) &=  \lambda^{\frac{1}{\alpha} (\frac{\gamma}{\gamma_0} + \gamma) -1- \gamma \tilde H_{\alpha,2}} g ((- \lambda^{\Gamma'_0} \mbu, \lambda^\Gamma \mbt- \lambda^{\Gamma'_0} \mbu ]) \\
	&=\lambda^{\frac{\gamma}{\gamma_0} - \gamma \chi q_2}
	\int_0^{t_1} \big\{
	\partial_1 g(\lambda^{\Gamma'_0} (\lambda' s \mbe_1 + t_2 \mbe_2 -\mbu))
	- \partial_1 g (\lambda^{\Gamma'_0} (\lambda' s \mbe_1 - \mbu ) ) \big\} \d s
	\end{align*}
	satisfies
	\begin{align*}
	f^\dagger_{\lambda} (\mbt; \mbu)
	&\sim \lambda^{\frac{\gamma}{\gamma_0} - \gamma \chi q_2}
	\int_0^{t_1} \big\{
	\partial_1 g_0 (\lambda^{\Gamma'_0} (\lambda' s \mbe_1 + t_2 \mbe_2 -\mbu))
	- \partial_1 g_0 (\lambda^{\Gamma'_0} (\lambda' s \mbe_1 - \mbu ) ) \big\} \d s\\
	&= \int_0^{t_1} \big\{
	\partial_1 g_0 (\lambda' s \mbe_1 + t_2 \mbe_2 -\mbu)
	- \partial_1 g_0 (\lambda' s \mbe_1 - \mbu ) \big\} \d s
	\sim t_1 \tilde h_2(t_2; \mbu)
	\end{align*}
	for almost all $\mbu \in \R^2_+$.
	To get the dominated function it is convenient to use $\chi q_2 = p_2(P-1)$ and $|\partial_1 g_0(\mbu)| \le C \rho_0(\mbu)^{\frac{1}{p_2} - 1}$, see \eqref{g12bdp}. Whence setting $\mbt = \1$, for
	$2^p > \rho_0 (\mbu) \ge 
	2^p \lambda'^{p_1}$,
	we get
	\begin{equation}\label{dctarg1}
	|f^\dagger_{\lambda} (\1; \mbu)| \le  C \int_0^1 (\rho_0( \lambda's \mbe_1 + \mbe_2 - \mbu)^{\frac{1}{p_2} - 1} +\rho_0(\lambda' s \mbe_1 - \mbu)^{\frac{1}{p_2} - 1} ) \d s,
	\end{equation}
	where $\rho_0 (\lambda' s \mbe_1 - \mbu)^{\frac{1}{p}} \ge \rho_0 (\mbu)^{\frac{1}{p}} - \rho_0(\lambda' s \mbe_1)^{\frac{1}{p}} \ge \frac{1}{2} \rho_0(\mbu)^{\frac{1}{p}}$ with $\int_{\rho_0(\mbu) < 2^p} \rho_0 (\mbu)^{\alpha(\frac{1}{p_2}-1)} \d \mbu < \infty$, see the proof of Proposition \ref{prelim}~(iii), and similar domination holds for the first term on the r.h.s.\ of \eqref{dctarg1}. For $\rho_0 (\mbu) \ge 2^p > \rho_0(\lambda^{\Gamma'_0} \mbu)$ we write
	\begin{align*}
	|f^\dagger_{\lambda} (\1; \mbu)| &=\lambda^{\gamma p_2}
	\Big|\int_{[\0,\1]} \partial_{12} g
	(\lambda^{\Gamma'_0} ( \lambda' s_1 \mbe_1 + s_2 \mbe_2 -\mbu ) ) \d \mbs \Big| \\
	&\le C \int_{[\0,\1]}\rho_0(\lambda' s_1 \mbe_1 + s_2 \mbe_2 - \mbu)^{-1} \d \mbs \le C \rho_0(\mbu)^{-1}
	\end{align*}
	with $\int_{\rho_0(\mbu) \ge 2^p} \rho_0 (\mbu)^{-\alpha} \d \mbu < \infty $, see the proof of Proposition \ref{prelim}~(iii). For $\rho_0(\mbu) < 2^p \lambda'^{p_1}$, we get
	\begin{align*}
	|f^\dagger_\lambda (\1;\mbu)| \le C (\lambda')^{-1} &(\rho_0(\lambda' \mbe_1+\mbe_2-\mbu)^{P-1} + \rho_0(\mbe_2-\mbu)^{P-1}\\
	&+ \rho_0(\lambda' \mbe_1-\mbu)^{P-1} + \rho_0(-\mbu)^{P-1}),
	\end{align*}
where	$ \int_{\rho_0(\mbu) <2^p \lambda'^{p_1}} \rho_0(\mbu)^{\alpha(P-1)} \d \mbu  = o(\lambda'^\alpha)$
	since $1+\frac{p_1}{p_2} + p_1 \alpha (P-1) > \alpha$ or $P_{\frac{1}{\alpha},\frac{1+\alpha}{\alpha}} > 1$, hence,
	$\int_{\rho_0(\mbu) <2^p \lambda'^{p_1}} |f^\dagger(\1;\mbu)|^\alpha \d \mbu = o (1)$.
	In the same manner,
	$\int_{\rho_0(\mbe_2-\mbu) <2^p \lambda'^{p_1}} |f^\dagger(\1;\mbu)|^\alpha \d \mbu = o (1)$.
	Finally, $\int_{\rho_0(\lambda^{\Gamma'_0} \mbu) \ge 2^p} |f_\lambda^\dagger (\1; \mbu)|^\alpha \d \mbu  = \lambda^{\alpha \gamma (1-\tilde H_{\alpha,2})} I_\lambda = o(1)$, since $1-\tilde H_{\alpha,2} =  \frac{p_2}{\alpha} (\alpha-P) >0$ and
$I_\lambda = \int_{\rho_0(\mbu) \ge 2^p} | \int_{[\0,\1]} \partial_{12} g (\lambda^{\Gamma} \mbs - \mbu) \d \mbs |^\alpha \d \mbu < C $
follows from Assumption (G)$^0_\alpha$ as in the two previous cases of $\gamma$.
	This proves \eqref{f3} and completes the proof of Theorem~\ref{mainthm}.
\end{proof}

Extending Theorem \ref{mainthm} we may ask the following two questions:  q1) what happens in the region  $P > \alpha $?  q2) when the (anisotropic) $\gamma$-rectangent limits in \eqref{Vlim}
agree with the L\'evy sheet $W_\alpha$ in \eqref{LSheet}?
(Obviously, the limits $V_\gamma $ in Theorem \ref{mainthm} \emph{never} agree with $W_\alpha $).

Concerning q1), recall Table \ref{tab1}, where in the parameter region $R_{11}$ the unbalanced limits $V_\pm $ of Theorem \ref{mainthm} are MSS RFs with indices $(H_1,H_2)\to (1,1)$ as $P \to \alpha $, the last fact an easy observation from formulas in \eqref{defH}.
One may expect that for $P> \alpha $ all scaling limits are $(1,1)$-MSS RFs of the form $\{t_1 t_2 V, \, (t_1,t_2) \in \R^2_+\}$, where $V $ is a r.v.
See \cite[Prop.~8.2.10]{samo2016} for related fact in the case of one-parameter processes.
As shown in Proposition \ref{propdif} this is true indeed under  weaker conditions on the infinitely divisible random measure $M$ which do not imply asymptotic stability.

We recall some facts about stochastic integrals w.r.t.\ $M$  whose characteristics $(\sigma,\nu)$ satisfy either Assumption (M)$_\alpha$ for $\alpha=2$, or $\sigma =0$, $\sup_{y>0}  y^\alpha \nu(\{u \in \R : |u|>y \}) <  \infty$ for some $0<\alpha<2$, moreover, $\nu$ is  symmetric for $\alpha=1$. Then
$M$ can be represented as $M(\d \mbu) = \sigma W_2(\d \mbu) +
\int_{\R} y \tilde N(\d \mbu, \d y) $
with $\tilde N = (N-n) I (1 < \alpha \le 2) + N I (0< \alpha < 1) + (N'-N'')I (\alpha = 1)$, where
$W_2$ is a Gaussian random measure on $\R^2$ with intensity $\d \mbu$,
$N$ is a Poisson random measure on $\R^2 \times \R$ with intensity $n (\d \mbu, \d y) = \d \mbu \nu (\d y)$, $N'$, $N''$ are Poisson random measures on $\R^2 \times \R$ with the same intensity $\frac{1}{2}n$,
$W_2$, 
$N$, $N'$, $N''$ are mutually independent.
Thus, stochastic integral w.r.t.\  $M$ can be defined as
$\int_{\R^2} f(\mbu)  M(\d \mbu) = \sigma \int_{\R^2} f(\mbu)  W_2(\d \mbu) + \int_{\R^2\times \R} f(\mbu) y \tilde N(\d \mbu, \d y)$,
where the last integral exists for any $f: \R^2 \to \R$ satisfying $\int_{\R^2} V (f(\mbu)) \d \mbu <\infty$, where
\begin{equation*} \label{fint}
V (u) =
\int_{\R} \big( (|uy| \wedge |uy|^2) I(1<\alpha\le 2) + (1 \wedge |uy|) I(0<\alpha<1) + (1 \wedge |uy|^2) I(\alpha= 1) \big) \nu (\d y)
\end{equation*}		
for $u \in \R$,
see \cite{rajp1989}, \cite[Sec.~3]{pils2014}.
Integrating the last expression by parts w.r.t.\ $y$ we obtain  that it 
does not exceed $C |u|^\alpha$ hence the integral  $\int_{\R^2} f(\mbu)  M(\d \mbu)$  is well-defined for any $f \in L_\alpha(\R^2)$.
Moreover, $f_\lambda \to f $  in $ L_\alpha(\R^2)$ implies $\int_{\R^2} f_\lambda (\mbu)  M(\d \mbu) \limd \int_{\R^2} f(\mbu)  M(\d \mbu) $.

\begin{proposition} \label{propdif}
	Let $X$ be a L\'evy driven RF given by \eqref{def:X}, where the L\'evy characteristics $(\sigma, \nu)$ of $M$ satisfy either Assumption (M)$_\alpha$ for $\alpha =2$ or $\sigma = 0 $, $\sup_{y>0}  y^\alpha \nu(\{u \in \R : |u|>y \}) <  \infty$ for some $0< \alpha < 2 $, moreover, $\nu$ is symmetric for $\alpha=1$. Let \eqref{G01} hold, where the kernel $g$ admits the $L_\alpha(\R^2)$-derivative $\partial_{12} g $ in the sense that
	\begin{equation}\label{gnormconv}
	\lim_{t_i \downarrow 0, \, i=1,2} \int_{\R^2} \big| (t_1 t_2)^{-1} g((\mbu, \mbu + \mbt]) - \partial_{12} g(\mbu)\big|^\alpha \d \mbu =  0.
	\end{equation}
	Then for any $\gamma >0$, as  $\lambda \downarrow 0$,
	\begin{equation*}
		\lambda^{-1-\gamma} X((\boldsymbol{0}, \lambda^\Gamma \mbt]) \limfdd t_1 t_2 V, \qquad \text{where }
		V := \int_{\R^2} \partial_{12} g(\mbu) M(\d \mbu). 
	\end{equation*}
\end{proposition}

\begin{proof}
	Follows from \eqref{Xinc1}, \eqref{gnormconv} and the criterion for the convergence of stochastic integrals before the proposition.
\end{proof}

Clearly,  \eqref{gnormconv} is not directly related to the asymptotic form of $g(\mbt)$ at $\0$ as specified in Assumption (G)$_\alpha$.
On the other hand if $g$ satisfies the latter assumption then from   \eqref{gg0}, \eqref{g12bd}, \eqref{g12bdp} and \eqref{intrho} we see that $\|\partial_{12} g\|_\alpha  < \infty $ and \eqref{gnormconv} hold when $P> \alpha$, relating Proposition \ref{propdif} to question q1).
Concerning question q2), in order that the rectangent limits are given by a L\'evy Sheet, we replace the power law behavior of $g$ in Assumption (G)$_\alpha$ by a condition that $g$ is bounded but discontinuous at $\boldsymbol{0}$.
Let $\R^2_{ij} := \{\mbt \in \R^2: \operatorname{sgn}(t_1) = i, \operatorname{sgn}(t_2) = j \}$, $i,j  = \pm 1$, denote the 4 open quadrants in $\R^2 $.

\medskip

\noi {\bf \emph{Assumption} (G)$^\prime_\alpha$.} \ The function $g$ has finite limits $g_{ij} = \lim_{|\mbt| \to 0,\,  \mbt \in \R^2_{ij}} g(\mbt)$, $i,j\in \{1,-1\}$, on each quadrant such that
\begin{equation}\label{g0}
	g[\boldsymbol{0}] := \sum_{i,j\in \{1,-1\}} ij g_{ij} \neq  0.
\end{equation}
Moreover,
\begin{equation}\label{gint}
	\int_{\R^2} \big| g((-\mbu, \mbt - \mbu]) -
	g[\boldsymbol{0}] I(\mbu \in (\boldsymbol{0}, \mbt])\big|^\alpha \d \mbu = o(t_1 t_2), \qquad t_i \downarrow 0, \ i=1,2.
\end{equation}

\medskip

Note that condition \eqref{g0} and the existence of the `limites quadrantales' $g_{ij}$ is typical for distribution functions in $\R^2 $  having an atom at $\boldsymbol{0}$ or more generally, functions from the Skorohod space $D(\R^2)$, which  are discontinuous at the origin, see \cite{bick1971}.
The most simple  example of  $g$ satisfying \eqref{g0} and \eqref{gint} is  the function taking constant values on each quadrant, viz.,
\begin{equation}\label{gind}
	g(\mbt) = \sum_{i,j\in \{1,-1\}} g_{ij} I( \mbt \in \R^2_{ij}), \qquad \mbt \in \R^2.
\end{equation}
For $g$ in \eqref{gind}, $ g ((-\mbu, \boldsymbol{t}-\mbu]) = g[\boldsymbol{0}] I( \mbu \in (\boldsymbol{0}, \mbt))$ implying
\begin{equation}\label{Xlevy0}
	\lambda^{- \frac{1}{\alpha}(1+\gamma)} X((\boldsymbol{0}, \lambda^\Gamma \mbt])
	=  g[\boldsymbol{0}] \lambda^{-\frac{1}{\alpha}(1+\gamma)} M((\boldsymbol{0}, \lambda^\Gamma \mbt))
	 \limfdd  g [\boldsymbol{0}] W_\alpha(\mbt)
\end{equation}
as in Remark \ref{rem3}.
A similar result holds for more general $g$ satisfying Assumption (G)$^\prime_\alpha$.

\begin{proposition} \label{proplevy}
	Let $X$ be a L\'evy driven RF in \eqref{def:X} satisfying \eqref{G01}, Assumptions 
	(G)$^\prime_\alpha$ and (M)$_\alpha$ for some $0 < \alpha \le 2 $.
	Then for any $\gamma >0$,
	\begin{equation}\label{Xlevy}
	\lambda^{-\frac{1}{\alpha}(1+\gamma)} X((\boldsymbol{0}, \lambda^\Gamma \mbt])
		\limfdd g[\boldsymbol{0}] W_\alpha(\mbt).
	\end{equation}
\end{proposition}

\begin{proof}
	Split $X ((\boldsymbol{0},\mbt]) = X'((\boldsymbol{0},\mbt]) + X''((\boldsymbol{0},\mbt])$, where $X'((\boldsymbol{0},\mbt]) := g [\boldsymbol{0}] M (\boldsymbol{0},  \mbt])$,  $X''((\boldsymbol{0},\mbt]) := \int_{\R^2}  (g((-\mbu, \mbt - \mbu]) - g [\boldsymbol{0}] I(\mbu \in (\boldsymbol{0}, \mbt]) ) M(\d \mbu)$.
	Then  $\lambda^{-\frac{1}{\alpha}(1+\gamma)} X'((\boldsymbol{0}, \lambda^\Gamma \mbt]) \limfdd g [\boldsymbol{0}] W_\alpha(\mbt) $ according to \eqref{Xlevy0}, while $\lambda^{-\frac{1}{\alpha}(1+\gamma)} X'' ((\boldsymbol{0}, \lambda^\Gamma \mbt])
	\limfdd  0 $ in view of condition \eqref{gint}.
\end{proof}

\section{Tangent limits of L\'evy driven fractional RF}
\label{sec6}

In this section we identify $\gamma$-tangent limits in \eqref{Tlim} of L\'evy driven fractional RF  in \eqref{Xinc2}, viz.,
\begin{equation}\label{X2}
X(\mbt) = \int_{\R^2} \big\{g (\boldsymbol{t} - \boldsymbol{u}) - g^0_{12}(-\mbu)\big\} M(\d \mbu), \qquad  \mbt \in \R^2,
\end{equation}
with stationary increments whose distribution depends on $g$ and the infinitely divisible random measure $M$.

\begin{theorem} \label{mainthm2}
	Let $X$ be a L\'evy driven fractional RF in \eqref{X2}, where $M$ satisfies Assumption (M)$_\alpha$, $0< \alpha \le 2$, and $\int_{\R^2} |g(\mbt - \mbu) - g^0_{12}(-\mbu)|^\alpha \d \mbu < \infty$ for all $\mbt \in \R^2$.
	Let $g(\mbt)= g_0(\mbt)(1 + o(1))$ as $|\mbt| \to 0$, where $g, g_0$ have partial derivatives $\partial_i$, $i=1,2$,  in \eqref{partialf} on $\R^2_0$  
	such that
	\begin{equation}\label{gtan}
	|g(\mbt)| + |g_0(\mbt)| \le C\rho(\mbt)^\chi, \qquad |\partial_i g(\mbt)| +
		|\partial_i g_0(\mbt)| \le C\rho(\mbt)^{\chi - \frac{1}{q_i}},  \quad  i=1,2,
	\end{equation}
	for all $\mbt \in \R_0^2$. Moreover, $g_0(\mbt) = \rho(\mbt)^\chi L(\mbt)$, $\mbt \in \R^2_0$, as in \eqref{glim}, where $L$ is a generalized invariant function and the parameters $\chi <0$, $q_i>0$, $i=1,2$, with $Q=\frac{1}{q_1}+\frac{1}{q_2}$ satisfy
	\begin{eqnarray}\label{qin2}
&		- \frac{1}{\alpha} Q < \chi < \frac{1}{q_1 \vee q_2} - \frac{1}{\alpha} Q.
	\end{eqnarray}
	Then the $\gamma$-tangent limits  $T_\gamma $ in \eqref{Tlim} of $X$  exist for any $\gamma >0$ under   normalization $d_{\lambda,\gamma} = \lambda^{H(\gamma)}$, $H(\gamma) := (1 \wedge \frac{\gamma}{\gamma_0}) H(\gamma_0)$, $H(\gamma_0) := \frac{1+\gamma_0}{\alpha} + \chi q_1$ and are given by
	\begin{equation}
		T_\gamma = \begin{cases} T_+, &\gamma > \gamma_0, \\
			T_-, &\gamma < \gamma_0, \\
			T_0, &\gamma = \gamma_0,
		\end{cases}
	\end{equation}
	where $\gamma_0 = \frac{q_1}{q_2}$, and
	\begin{equation}\label{T2}
	T_0(\mbt) :=  \int_{\R^2} \big\{g_0(\boldsymbol{t} - \boldsymbol{u}) - g_0(-\mbu)\big\} W_\alpha (\d \mbu), \qquad  \mbt \in \R^2_+,
	\end{equation}
	and $W_\alpha$ is $\alpha$-stable random measure as in Theorem \ref{mainthm}, and the `unbalanced' $\gamma$-tangent RFs
	\begin{equation*}
	T_+ (\mbt) := T_0(t_1,0),  \qquad T_- (\mbt) := T_0(0,t_2), \qquad \mbt = (t_1,t_2)\in \R^2_+,
	\end{equation*}
	depend on only one coordinate on the plane.
\end{theorem}

\begin{remark}
	We can rewrite the condition \eqref{qin2} in terms of $p_i$, $i=1,2$, in \eqref{pp} as $P_{\frac{1}{\alpha}, \frac{1+\alpha}{\alpha}}<1$, $P_{\frac{1+\alpha}{\alpha},\frac{1}{\alpha}}<1$, $P> \frac{\alpha}{1+\alpha}$, which corresponds to parameter region $R_{22}$ in Figure \ref{fig}.
\end{remark}

\begin{remark}
	For $\alpha=2$, in Example 3.1 we have $g(\mbt) = g_0(\mbt) = \norm \mbt\norm^{H-1}$ with $q_1 = q_2 = 2$, $Q=1$, $\chi = \frac{H-1}{2}$ and conditions \eqref{qin2} are equivalent $\frac{1-H}{2} < \frac{1}{2} < 1 - \frac{H}{2} $ or $0<H<1$.
	Condition \eqref{gtan} in this example is also satisfied for all $\mbt \in \R^2_0$ since $\partial_i  \norm \mbt\norm^{H-1} = (H-1) \norm \mbt \norm^{H-2} \partial_i\norm \mbt \norm$ with $\partial_i \norm \mbt \norm = \frac{t_i}{\norm \mbt \norm}$, $i=1,2 $.
\end{remark}

\begin{proof}
Let us check that the stochastic integral in \eqref{T2}  is well-defined for any $\mbt \in \R^2 $ or $\|h_0 (\mbt; \cdot) \|_\alpha < \infty $, where $h_0(\mbt;  \mbu) := g_0(\boldsymbol{t} - \boldsymbol{u}) - g_0(-\mbu)$.
	  Let $q:=\max\{q_1,q_2,1\}$.
	Then using \eqref{gtan}, we get $\int_{\rho(\mbu) < 2^q \rho(\mbt)} |h_0(\mbt;\mbu)|^\alpha \d \mbu \le C \int_{\rho(\mbu) < 3^q \rho(\mbt)} \rho(\mbu)^{\alpha \chi} \d \mbu < \infty $ as $Q > - \alpha \chi $, see \eqref{intrho}.
	Next, for $\rho(\mbu)\ge 2^q\rho(\mbt)$ we rewrite $h_0(\mbt; \mbu)=  \int_0^{t_1} \partial_1 g_0 (s -u_1,t_2-u_2) \d s + \int_0^{ t_2} \partial_2 g_0 (-u_1, s-u_2) \d s$, where $\partial_1 g_0 (s -u_1,t_2-u_2) \le C \rho(\mbu)^{\chi - \frac{1}{q_1}}$, $\partial_2 g_0 (-u_1,s-u_2) \le C \rho(\mbu)^{\chi - \frac{1}{q_1}}$ by \eqref{gtan}, \eqref{rhoineq2}. We get $\int_{\rho(\mbu)\ge 2^q\rho(\mbt)} |h_0(\mbt; \mbu)|^\alpha \d \mbu \le C \int_{\rho(\mbu) \ge 2^q \rho(\mbt)} \rho(\mbu)^{\alpha(\chi -  \frac{1}{q_1 \vee q_2})}\d \mbu < \infty$.
Hence, $\|h_0 (\mbt; \cdot) \|_\alpha < \infty $.

	To prove the convergence in \eqref{Tlim} we use Proposition \ref{offp}.
	It suffices to consider the case $\mbt_0 = \0$.
	Let first $\gamma = \gamma_0$ with $H(\gamma_0) = \frac{1}{\alpha}(1 + \gamma_0) +  \chi q_1 $.
	Then $\lambda^{-H(\gamma_0)} X(\lambda^{\Gamma_0}\mbt) = \int_{\R^2} f_\lambda (\mbt; \mbu) M(\d \mbu)$,  where $f_\lambda (\mbt; \mbu) := \lambda^{-H(\gamma_0)} (g( \lambda^{\Gamma_0} \boldsymbol{t} - \boldsymbol{u}) - g(-\mbu)) $ satisfies
	$f^\dagger_\lambda (\mbt; \mbu) = \lambda^{-\chi q_1} (g( \lambda^{\Gamma_0} (\boldsymbol{t} - \boldsymbol{u})) - g(- \lambda^{\Gamma_0}\mbu)) \to h_0(\mbt;  \mbu)$
	for all $\mbu \neq \0, \mbt $ according to $g(\mbt) = g_0(\mbt)(1+o(1))$, $|\mbt| \to 0$.
	The domination argument implying $f^\dagger_\lambda (\mbt; \cdot) \to h_0(\mbt; \cdot)$ in $L_\alpha (\R^2)$ uses \eqref{gtan} and follows as in the proof $\|h_0 (\mbt; \cdot) \|_\alpha < \infty $ above.

	Next, let $\gamma < \gamma_0$ with $H(\gamma) = \frac{\gamma}{\gamma_0} H(\gamma_0) = \frac{1}{\alpha}(\gamma + \frac{\gamma}{\gamma_0}) + \gamma \chi q_2$.
	Then we have $\lambda^{-H(\gamma)} X(\lambda^{\Gamma}\mbt) = \int_{\R^2} f_{\lambda,\gamma} (\mbt;  \mbu) M(\d \mbu)$, where using Proposition \ref{offp} with $\mu_1 = \frac{\gamma}{\gamma_0}$, $\mu_2 = \gamma$, we see that
	$f^\dagger_{\lambda,\gamma} (\mbt; \mbu) = \lambda^{-\gamma \chi q_2} (g( \lambda t_1 - \lambda^{\frac{\gamma}{\gamma_0}}  u_1, \lambda^\gamma t_2 - \lambda^\gamma  u_2)  - g(-\lambda^{\frac{\gamma}{\gamma_0}} u_1, - \lambda^\gamma u_2 )) \to h_0((0,t_2); \mbu)$
	for all $\mbu \neq \0, (0,t_2)$.
	The $L_\alpha(\R^2)$-convergence $f^\dagger_{\lambda,\gamma} (\mbt; \cdot) \to h_0((0,t_2); \cdot)$ in the case  $\gamma < \gamma_0$ and the proof of \eqref{Tlim} for  $\gamma > \gamma_0$ follows analogously and we omit the details.
	Theorem \ref{mainthm2} is proved.
\end{proof}

\section{Concluding comments}
\label{sec7}

	\noi 1.  Unbalanced scaling ($\gamma \neq \gamma_0$)
	of fractional RF on $\R^2$ may lead to degenerate dependence in one direction (either horizontal or vertical).
	The critical  $\gamma_0$ coincides with the `intrinsic local dependence ratio'
	$\frac{q_1}{q_2}= \frac{p_1}{p_2}$ of the RF, defined in terms the exponents $\chi$, $q_i$, $i=1,2$, 
	in the asymptotic form \eqref{glim} of the moving-average kernel $g(\mbt)$ at $\mbt = \0$. The degeneration
	is apparent in $(H_1, H_2)$-MSS property of rectangent limits with one of $H_i, i=1,2$,
	equal 1 or 0, indicating either extreme positive  ($H_i = 1$) or extreme negative ($H_i = 0$) dependence of the limit RF in
	direction $\mbe_i$, $i=1,2 $. The above facts are very important for statistical estimation of $p_i$, $i=1,2$, using power variations of
	rectangular increments on a dense grid (see Comment 5 below).

\medskip

\noi 2.
We believe that our results can be extended to fractional type infinitely divisible RFs $X$ on $\R^d$, $d \ge 3 $.
However, the results in \cite{bier2017, sur2019} on  large-scale anisotropic scaling 
suggest that the class of the limit rectangent RFs in higher dimensions $d\ge 3$ is more complex  and its complete description can be difficult.
On the other side, it would be very interesting to extend the class of two-parameter L\'evy driven fractional RFs in \eqref{def:X} by including a \emph{random volatility RF} (independent of $M$) as in the case of \emph{ambit RF} \cite{barn2018}.

\medskip

\noi 3.
The $\gamma$-tangent and $\gamma$-rectangent limits in our paper are defined through finite-dimensional convergence only.
This raises the question of a functional convergence in these limits and also about  the  path properties of the limit tangent and rectangent RFs.
While in some cases these limits are a.s.\ continuous and classical, in other cases they are extremely singular, see Remark \ref{rem2}, suggesting that a functional convergence is not feasible in such cases.


\medskip

\noi 4. As noted by a referee, the scaling in \eqref{Tlim}, \eqref{Vlim} with diagonal matrix $\Gamma $ is  quite particular, raising
the question about these limits for general $2\times 2$ matrix $\Gamma $ with positive
real parts  of the eigenvalues, c.f.\ \cite{bier2007}.  This question is interesting and open.  We expect that
a possible answer depends on   $\Gamma $ or on the way  the point $\lambda^\Gamma \mbt$  tends to $\0$ as $\lambda \to 0$.
If the trajectory $\{\lambda^\Gamma \mbt, \, 0< \lambda < 1 \}$
winds up around $\0$ infinitely many times,
the limit in \eqref{Vlim} under the premises of Theorem 
\ref{mainthm}  probably do not exist.  On the other hand,
if the trajectory of $\lambda^\Gamma \mbt$
approaches  $\0$ along
some `oblique' direction,  the  limit in \eqref{Vlim} 
is more likely to exist since this case
resembles the scaling of  RF with `oblique' dependence axis \cite{pils2020}.

\medskip

\noi 5.
We expect that our results can be applied for statistical estimation and identification
of fractional parameters $p_i$, $i=1,2$, based on observations of $X$ in \eqref{def:X} on a dense rectangular grid of $[0,1]^2$. Estimation of fractional parameters (`Hurst estimation') is an important part of statistical inference for stochastic processes with one-dimensional time. In RF context, results on Hurst estimation
are less developed, with most studies limited to
parametric and/or Gaussian models
\cite{cohen2013, lee2021, loh2015, pakk2014, pakk2016}.
For a truncated $\alpha$-stable (isotropic) counterpart to fractional Brownian RF in Example 3.1 this question is treated in \cite{benn2004, cohen2012} using log-variations of 
(second-order) increments of $X$ on a dyadic grid. Using power variations of increments of $X$ on a dense square grid,
\cite{bass2020} discuss Hurst estimation for a general class of $\alpha$-stable fractional isotropic RFs. 
Obviously, in the case of anisotropic RF the number of fractional parameters is more than 1 and a fixed
	(square or non-square) observation grid seems  insufficient. Indeed, power variations computed from a grid of size $(\frac{1}{n}, \frac{1}{n^\gamma})$ for a fractional RF $X$
	as in this paper are supposed to estimate the normalizing exponent $H(\gamma)$ in Theorem  \ref{mainthm}, which does not determine $p_1 \neq p_2$ and takes a different form in different parameter regions in Figure \ref{fig}. It seems that consistent estimation of $p_i$, $i=1,2$, requires
	power variations from {\it two} grids corresponding to $\gamma', \gamma''$ such that $\gamma'<\gamma_0<\gamma''$, which should estimate the Hurst parameters $H_{\alpha,i}$, $\tilde H_{\alpha, i}$, $i=1,2$, in \eqref{defH} leading to 
	estimates
	of $p_i$, $i=1,2$.
	Finally, the question about fluctuations of such estimators or the second order asymptotics  of power variations
	(in spirit of \cite{bass2017})
	seems to be strongly
	related to the dependence structure of rectangent limits exposed in Theorem  \ref{mainthm}.

%
%

\section*{Acknowledgement}

The authors thank two anonymous referees and the AE for useful comments. VP
acknowledges financial support from the project `Ambit fields: probabilistic properties and statistical inference' funded by Villum Fonden.
Also, VP gratefully acknowledges financial support of ERC Consolidator Grant 815703 `STAMFORD: Statistical Methods for High Dimensional Diffusions'.


\end{document}